\documentclass[pdflatex,sn-mathphys-num]{sn-jnl}

\usepackage{amsmath,amssymb,amsfonts}%
\usepackage{amsthm}%
\usepackage{xcolor}%
\usepackage{bbm}

\usepackage{geometry} 

\theoremstyle{thmstyleone}%
\newtheorem{theorem}{Theorem}
\newtheorem{lemma}{Lemma}%
\newtheorem{proposition}[theorem]{Proposition}%

\theoremstyle{thmstyletwo}%
\newtheorem{example}{Example}%
\newtheorem{remark}{Remark}%

\theoremstyle{thmstylethree}%

\newcommand{\me}{\mathbb{E}}
\newcommand{\mr}{\mathbb{R}}

\newcommand{\mn}{\mathbb{N}}
\newcommand{\mmp}{\mathbb{P}}

\DeclareMathOperator{\1}{\mathbbm{1}}
\newcommand{\eee}{{\rm e}}

\begin{document}

\title{On local large deviations for decoupled random walks}


\author[1]{\fnm{Dariusz} \sur{Buraczewski}}\email{dariusz.buraczewski@uwr.edu.pl}

\author[2]{\fnm{Alexander} \sur{Iksanov}}\email{iksan@univ.kiev.ua}

\author*[3]{\fnm{Alexander} \sur{Marynych}}\email{omarynych@kse.org.ua}

\affil[1]{\orgdiv{Mathematical Institute}, \orgname{University of Wroc{\l}aw}, \orgaddress{\street{Pl. Grunwaldzki 2/4}, \city{Wroclaw}, \postcode{50384}, \country{Poland}}}

\affil[2]{\orgdiv{Faculty of Computer Science and Cybernetics}, \orgname{Taras Shevchenko National University of Kyiv}, \orgaddress{\street{Volodymyrska 60}, \city{Kyiv}, \postcode{01601}, \country{Ukraine}}}

\affil[3]{\orgdiv{Department of Mathematics}, \orgname{Kyiv School of Economics}, \orgaddress{\street{Mykoly Shpaka 3}, \city{Kyiv}, \postcode{03113}, \country{Ukraine}}}


\abstract{A decoupled standard random walk is a sequence of independent
random variables $(\hat{S}_n)_{n \geq 1}$ such that, for each $n \geq
1$, the distribution of $\hat{S}_n$ is the same as that of $S_n = \xi_1
+ \ldots + \xi_n$, where $(\xi_k)_{k \geq 1}$ are independent copies of
a nonnegative random variable $\xi$. We consider the counting process
$(\hat{N}(t))_{t\geq 0}$ defined as the number of terms $\hat{S}_n$ in
the sequence $(\hat{S}_n)_{n \geq 1}$ that lie within the interval $[0,
t]$. Under various assumptions on the tail distribution of $\xi$, we
derive logarithmic asymptotics for the local large deviation probabilities
$\mathbb{P}\{\hat{N}(t) = \lfloor b \, \mathbb{E}[\hat{N}(t)]
\rfloor\}$ as $t \to \infty$ for a fixed constant $b > 0$. These
results are then applied to obtain a logarithmic local large deviations asymptotic
for the counting process associated with the infinite Ginibre ensemble
and, more generally, for determinantal point processes with the
Mittag-Leffler kernel.}



\maketitle

\section{Introduction}

Let $\xi_1$, $\xi_2,\ldots$ be independent copies of a nonnegative random variable $\xi$ with a nondegenerate distribution. Put $S_n=\xi_1+\ldots+\xi_n$ for $n\in\mn$ and then $N(t):=\sum_{n\geq 1}\1_{\{S_n\leq t\}}$ for $t\geq 0$. The random sequence $(S_n)_{n\geq 1}$ is called {\it standard random walk} with nonnegative jumps and the random process $(N(t))_{t\geq 0}$ is called {\it renewal process}. Let $\hat S_1$, $\hat S_2,\ldots$ be independent random variables such that, for each $n\in\mn$, $\hat S_n$ has the same distribution as $S_n$. Put $\hat N(t):=\sum_{n\geq 1}\1_{\{\hat S_n\leq t\}}$ for $t\geq 0$. Following \cite{Alsmeyer+Iksanov+Kabluchko:2025}, we call the sequence $(\hat S_n)_{n\geq 1}$ {\it decoupled standard random walk} and the process $(\hat N(t))_{t\geq 0}$ {\it decoupled renewal process}. The distribution of $\hat S_1$ coincides with the distribution of $\xi$. We shall interchangeably refer to the distribution of $\hat S_1$ or the distribution of $\xi$ depending on convenience.

Now we explain the reason behind introducing these decoupled processes. Let $\Theta$ be an {\it infinite Ginibre point process} on $\mathbb{C}$ (the set of complex numbers), that is, a simple point process such that, for any $k\in\mn$ and any mutually disjoint Borel subsets $B_1,\ldots, B_k$ of $\mathbb{C}$ $$\me \Big[ \prod_{j=1}^k \Theta(B_j)\Big]=\int_{B_1\times\ldots\times B_k} {\rm det}(C(z_i, z_j))_{1\leq i,j\leq k}\,{\rm d}z_1\ldots {\rm d}z_k.$$ Here, ${\rm det}$ denotes the determinant, $C(u,w)=\pi^{-1}\eee^{u\bar w-|u|^2/2-|w|^2/2}$ for $u,w\in \mathbb{C}$ ($\bar w$ is the complex conjugate of $w$). Thus, $\Theta$ is a {\it determinantal point process} with kernel $C$ with respect to Lebesgue measure on $\mathbb{C}$. We refer to the book~\cite{Hough+Krishnapur+Peres+Virag:2009}, which contains a wealth of information on the determinantal point processes. A discussion of the Ginibre point processes can be found in Sections 4.3.7 and 4.7 of that book and in Part I of the very recent monograph \cite{Byun+Forrester:2025}.

A link between the infinite Ginibre point processes and the decoupled standard random walks is established with the help of the infinite version of Kostlan's result \cite{Kostlan:1992}, stated in Theorem 1.1 of~\cite{Fenzl+Lambert:2022}. The result says that the set of absolute values of atoms of $\Theta$ has the same distribution as $((\hat S_n)^{1/2})_{n\geq 1}$, where
\begin{equation}\label{eq:exponential}
\hat S_1~ \text{has the exponential distribution of unit mean}.
\end{equation}
The infinite Ginibre point process is not the only point process on $\mathbb{C}$ giving rise to a decoupled standard random walk. For $\rho>0$, define the kernel $C_\rho$ by $$C_\rho(u,w)=\frac{\rho}{2\pi}{\rm E}_{2/\rho,\,2/\rho}(u\bar w)\eee^{-|u|^\rho/2-|w|^\rho/2},\quad u,w\in \mathbb{C}.$$ Here, for $a,b>0$, ${\rm E}_{a,\,b}$ denotes the Mittag-Leffler function with parameters $a$ and $b$ given by $${\rm E}_{a,\,b}(z):=\sum_{k\geq 0}\frac{z^k}{\Gamma(ak+b)},\quad z\in \mathbb{C},$$ and $\Gamma$ is the Euler gamma-function. It is stated on pp.~3-4 in \cite{Adhikari:2018} that the set of absolute values of atoms of a determinantal point process $\Theta_\rho$ with kernel $C_\rho$ with respect to Lebesgue measure on $\mathbb{C}$ has the same distribution as $((\hat S_n)^{1/\rho})_{n\geq 1}$, where $\hat S_1$ has the gamma distribution with parameters $2/\rho$ and $1$, that is,
\begin{equation}\label{eq:gamma}
\mmp\{\hat S_1\in{\rm d}x\}=\frac{1}{\Gamma(2/\rho)}x^{2/\rho-1}\eee^{-x}\1_{(0,\infty)}(x){\rm d}x.
\end{equation}
Here and in what follows, $x\mapsto\1_A(x)$ denotes the indicator function, defined by $\1_{A}(x)=1$ if $x\in A$ and $\1_{A}(x)=0$ if $x\not\in A$. Plainly, the case $\rho=2$ corresponds to the infinite Ginibre point process. For each $t\geq 0$, let $\Theta_\rho(D_t)$ denote the number of atoms of $\Theta_\rho$ inside the disk $D_t:=\{z\in \mathbb{C}: |z|<t\}$. Then
\begin{equation}\label{eq:principal}
(\Theta_\rho(D_t))_{t\geq 0}~~\text{has the same distribution as}~~(\hat N(t^\rho))_{t\geq 0}=\Big(\sum_{n\geq 1}\1_{\{\hat S_n\leq t^\rho\}}\Big)_{t\geq 0},
\end{equation}
with $\hat S_1$ as in \eqref{eq:gamma}.

Further, according to Theorem 3.1 in \cite{Akemann+Strahov:2013}, the set of absolute values of atoms of a generalized infinite Ginibre point process parameterized by $m\in\mn$ has the same distribution as $((\hat S^{(1)}_n\hat S^{(2)}_{n}\cdot\ldots\cdot \hat S^{(m)}_n)^{1/2})_{n\geq 1}$, where $(\hat S^{(1)}_n)_{n\geq 1},\ldots, (\hat S^{(m)}_n)_{n\geq 1}$ are independent copies of $(\hat S_n)_{n\geq 1}$, with $\hat S_1$ as in \eqref{eq:exponential}.

In \cite{Alsmeyer+Iksanov+Kabluchko:2025}, the asymptotic behavior of $\log \mmp\{\hat N(t)=0\}=\log \mmp\{\min_{n\geq 1}\,\hat S_n>t\}$ as $t\to\infty$ was found under various assumptions imposed on the distribution of $\xi$, both in the heavy- and light-tailed scenarios. Under \eqref{eq:exponential}, the asymptotic behavior of $\log \mmp\{\hat N(t)=\lfloor b\me [\hat N(t)]\rfloor\}$ as $t\to\infty$ for $b\geq 0$ was investigated 
in Theorem 1.1 of \cite{Shirai:2006}. Our first purpose is to extend the two aforementioned results. Namely, we aim at determining the asymptotic behavior of $\log \mmp\{\hat N(t)=\lfloor b\me [\hat N(t)]\rfloor\}$ as $t\to\infty$ for $b>0$, $b\neq 1$ under the same set of assumptions imposed on the distribution of $\xi$ as in \cite{Alsmeyer+Iksanov+Kabluchko:2025}. Thus, we are concerned with the {\it local large deviations} of $\hat N(t)$. The case $b=1$ belongs to the realm of local central limit theorems. Although the corresponding result is less interesting than that in the case $b\neq 1$, it will be given for completeness.

Under \eqref{eq:exponential}, it was shown in \cite{Forrester:1992} that
\begin{equation}\label{eq:forrest}
\mmp\{\hat N(t)=0\}=\exp\big(-t^2/4-(t\log t)/2+(1-(\log 2\pi)/2)t+ c t^{1/2}+o(t^{1/2})\big),\quad t\to\infty,
\end{equation}
where the constant $c$ is defined by a rather complicated convergent integral (see also Theorem 1.7 in~\cite{Charlier:2024}, with $g=1$, $\alpha=0$, $b=1$ and $r_2\in (0,1)$, where a few additional terms of the expansion were derived). The first three terms of the latter limit relation were also obtained in formula (4.5) of~\cite{Akemann+Philips+Shifrin:2009} via a different argument. Our second purpose is, for a wide class of distributions of $\xi$, which particularly covers the gamma distributions as in \eqref{eq:gamma},
to provide two terms of an asymptotic expansion like~\eqref{eq:forrest} for $\mmp\{\hat N(t)=\lfloor b\me [\hat N(t)]\rfloor\}$ as $t\to\infty$ for $b>0$, $b\neq 1$. Finding more terms of an expansion is beyond our reach at the moment.

\section{Main results}

\subsection{Local large deviations}

As usual, $f(t)\sim g(t)$ as $t\to A$ will mean that $\lim_{t\to A}(f(t)/g(t))=1$.

We have to distinguish the cases $\mu:=\me [\xi]=\infty$ and $\mu<\infty$. To begin with, we treat a subcase of the former case. Assume that
\begin{equation}\label{eq:domain alpha}
\mmp\{\xi>t\}~ \sim~ t^{-\alpha}\ell(t),\quad t\to\infty
\end{equation}
for some $\alpha\in [0,1)$ and some $\ell$ slowly varying at $\infty$, which particularly ensures that $\mu=\infty$. We refer the reader to Chapter I of~\cite{Bingham+Goldie+Teugels:1989} for the definitions and a comprehensive discussion of slowly and regularly varying functions.

If $\alpha\in (0,1)$, let $(W_\alpha(t))_{t\geq 0}$ denote a drift-free $\alpha$-stable subordinator with $$-\log \me \big[ \exp(-zW_\alpha(t))\big] = \Gamma(1-\alpha) tz^\alpha,\quad z\geq 0,$$ where $\Gamma$ is the Euler gamma-function. Put $W_\alpha^\leftarrow(t):=\inf\{s\geq 0: W_\alpha(s)>t\}$ for $t\geq 0$. The process $(W_\alpha^\leftarrow(t))_{t\geq 0}$ is called an
inverse $\alpha$-stable subordinator. It can be checked that
\begin{equation*} 
\me \big[\exp\big( s\Gamma(1-\alpha)W_\alpha^\leftarrow(1)\big)\big]
=\sum_{n\geq
0}\frac{s^n}{\Gamma(\alpha n+1)}={\rm E}_{\alpha,\, 1}(s),\quad s\geq 0,
\end{equation*}
which particularly shows that the variable $W_\alpha^\leftarrow(1)$ has finite power moments of all positive orders and that $\me [W_\alpha^\leftarrow(1)]=(\Gamma(1-\alpha)\Gamma(1+\alpha))^{-1}$. Put $Z_\alpha:=W_\alpha^\leftarrow(1)/\me [W_\alpha^\leftarrow(1)]$.

Assume now that $\alpha=0$ and denote by $Z_0$ a random variable with the exponential distribution of unit mean. Just in case, we note that neither a $0$-stable subordinator, nor an inverse $0$-stable subordinator exist.

For $\alpha\in [0,1)$, define the function $f_\alpha$ by $$f_\alpha(s):=\int_0^\infty\log\big((\eee^s-1)\mmp\{Z_\alpha>y\}+1\big){\rm d}y,\quad s\in\mr.$$
Observe that, for each $y>0$, the integrand 
is strictly convex in $s$ as the cumulant generating function of a Bernoulli distribution with success probability $\mathbb{P}\{Z_{\alpha}>y\}\in (0,1)$. Thus, $f_\alpha$ is strictly convex as a mixture of strictly convex functions.
Put $U(t):=\me [\hat N(t)]=\me [N(t)]$ for $t\geq 0$. The function $U+1$ is called {\it renewal function}.
\begin{theorem}\label{thm:infinite deviation}
Let $b>0$, $b\neq 1$. Suppose~\eqref{eq:domain alpha} with $\alpha\in [0,1)$. Then
\begin{equation}\label{eq:thm1}
\lim_{t\to\infty}\frac{-\log \mmp\{\hat N(t)=\lfloor b\, U(t)\rfloor\}}{U(t)}=J_\alpha(b)\in (0,\infty).
\end{equation}
Here, $J_\alpha$ is the Legendre transform of $f_\alpha$ given by $$J_\alpha(x):=\sup_{s\in\mr}\,(sx-f_\alpha(s)),\quad x>0$$ and
\begin{equation}\label{eq:ren asymp}
U(t)~\sim~\frac{1}{\Gamma(1-\alpha)\Gamma(1+\alpha)}\frac{t^\alpha}{\ell(t)},\quad t\to\infty.
\end{equation}
\end{theorem}
\begin{remark}\label{rem:infinite}
Since $J_\alpha(1)=0$, the limit relation of Theorem \ref{thm:infinite deviation} holds true when formally putting $b=1$ in it. However, this fact does not follow from our proof of Theorem \ref{thm:infinite deviation}. Rather, it is a consequence of the local central limit theorem (Proposition \ref{thm:local limit}). It was proved in Theorem 3.5 of \cite{Alsmeyer+Iksanov+Kabluchko:2025} that, for $\alpha\in (0,1)$, $\lim_{t\to\infty}\mmp\{\xi>t\}(-\log \mmp\{\hat N(t)=0\})=\int_0^\infty (-\log \mmp\{W_\alpha^\leftarrow(1)\leq y\}){\rm d}y$. In view of \eqref{eq:ren asymp}, this is equivalent to
\begin{equation}\label{eq:inter1}
\lim_{t\to\infty}\frac{-\log \mmp\{\hat N(t)=0\}}{U(t)}=\int_0^\infty (-\log \mmp\{Z_\alpha\leq y\}){\rm d}y.
\end{equation}
It will be shown in the proof of Theorem \ref{thm:infinite deviation} that $f_\alpha$ is a continuously differentiable function. This in combination with strict convexity of $f_\alpha$ ensures that
$J_\alpha(b)=bs_b-f_\alpha(s_b)$, where $s_b$ is the unique solution to the equation $f_\alpha^\prime(s)=b$. It can be checked that $\lim_{b\to 0+}s_b=-\infty$ and $\lim_{b\to 0+} bs_b=0$. With this at hand, relation \eqref{eq:inter1} can be obtained by formally sending $b\to 0+$ on both sides of \eqref{eq:thm1}. Of course, this argument is non-rigorous and should be deemed as such.
\end{remark}

We now turn to the situations in which $\mu < \infty$. The logarithmic local large deviations probability exhibits a universal growth rate $t^2$ for $b>1$ only, see Theorem~\ref{thm: light deviation2} below. For $b\in (0,1)$, it depends on the heaviness of the distribution tail of $\xi$. Specifically:
\begin{itemize}
    \item If the tail is \emph{very heavy}, that is, it varies regularly at infinity of a negative 
    finite index, the growth rate is $t\log t$, see Theorem~\ref{thm: regular deviation}.
    \item If the tail is \emph{moderately heavy}, that is $t\mapsto -\log \mathbb{P}\{\xi > t\}$ varies regularly of a positive index smaller than one, 
    the growth rate is $-t\log\mmp\{\xi>t\}$, see Theorem~\ref{thm:semi}.
    \item If the tail is \emph{light}, that is, $\mathbb{E}[\eee^{s\xi}] < \infty$ for some $s > 0$, the growth rate is $t^2$, see Theorem~\ref{thm: light deviation}.
\end{itemize}
The distinction between the cases $b\in (0,1)$ and $b > 1$ can be roughly explained as follows. When $b>1$, the large deviations event $\{\hat N(t) = \lfloor b\mu^{-1}t \rfloor\}$ arises from an unusually large
number of events $\{\hat S_n \leq t\}$ that occur. This happens when the sequence $(\hat S_n)_{n \geq 1}$ stays below its expected value. The behavior in this regime is governed by the distribution left tail of $\xi$. This tail is well-controlled due to the standing assumption $\xi \geq 0$, for instance, with the help of the large $s$ behavior of $s\mapsto \me [\eee^{-s\xi}]$.  When $b\in (0,1)$, the event $\{\hat N(t) = \lfloor b\mu^{-1}t \rfloor\}$ is driven by an unusually small number of events $\{\hat S_n \leq t\}$ that occur. The latter is caused by the fact that the sequence $(\hat S_n)_{n\geq 1}$ stays
above its expected value. The behavior in the case $b\in(0,1)$ is thus dictated by the right distribution tail of
$\xi$, leading to the aforementioned trichotomy.

\begin{theorem}\label{thm: regular deviation}
Let $b\in (0,1)$. Assume that relation \eqref{eq:domain alpha} holds for some $\alpha>1$. Then
$$\lim_{t\to\infty}\frac{-\log \mmp\{\hat N(t)=\lfloor b\mu^{-1}t\rfloor \}}{t\log t}=\frac{(\alpha-1)(1-b)}{\mu},$$ where $\mu=\me [\xi]<\infty$.
\end{theorem}

\begin{theorem}\label{thm:semi}
Let $b\in (0,1)$. Assume that, for some $\varrho\in (0,1)$ and some $\ell$ slowly varying at $\infty$, $H(t):=-\log \mmp\{\xi>t\}=t^\varrho \ell(t)$ for $t>0$ and that
\begin{equation}\label{eq:borov}
H(t+v(t))-H(t)=\varrho v(t)t^{-1}H(t)+o(v(t)t^{-1}H(t))+o(1),\quad t\to\infty
\end{equation}
for any function $v$ satisfying $v(t)=o(t)$ as $t\to\infty$. Then $$\lim_{t\to\infty} \frac{-\log \mmp\{\hat N(t)=\lfloor b\mu^{-1}t\rfloor \}}{t^{\varrho+1}\ell(t)}=\frac{(1-b)^{\varrho+1}}{\mu (\varrho+1)},$$ where $\mu=\me[\xi]<\infty$.
\end{theorem}
\begin{remark}
Condition~\eqref{eq:borov} is not as restrictive as it appears. Some sufficient conditions for~\eqref{eq:borov} can be found on p.~991 in \cite{Borovkov+Mogulskii:2006}. For example, if $\ell$ is eventually differentiable and $\ell'(t)=o(t^{-1}\ell(t))$ as $t\to\infty$, then~\eqref{eq:borov} holds.
\end{remark}

Let $\Lambda$ be the moment generating function of $\xi$, that is, $\Lambda(s) = \me [\eee^{s\xi}]$ for $s\in\mr$. Plainly, $\Lambda(s)<\infty$ for all $s\leq 0$. Assume that $\Lambda(s)<\infty$ for some $s>0$, which particularly entails that $\mu=\me [\xi]<\infty$. It is known, see Lemma 2.2.5 in~\cite{Dembo+Zeitouni:2009}, that in this case
the Legendre transform $I$ of the convex function $s\mapsto \log \Lambda(s)$ satisfies
$$
I(x)=\sup_{s\geq 0}\,(sx-\log \Lambda(s))=\sup_{s\in D}\,(sx-\log \Lambda(s)),\quad x\geq \mu,
$$
where $D:=\{s\geq 0\,:\,\Lambda(s)<\infty\}$. Put $m(s) = \Lambda^\prime(s)/\Lambda(s)$ for $s\in D$ and
$B:=\sup D$. The function $m$ is strictly increasing and continuous on $(0,B)$, and the limit $A_0 = \lim_{s\to B-} m(s)$ may be infinite (see Lemma on p.~287 in~\cite{Petrov:1965}).

Recall that the distribution of $\xi$ is nonlattice if it is not concentrated on any lattice $(s_1k+s_2)_{k\geq 0}$ with $s_1>0$ and $s_2\geq 0$.

\begin{theorem}\label{thm: light deviation}
Let $b\in (0,1)$. Assume that $\mmp\{\xi\leq x\}<1$ for all $x\geq0$ and $\me [\eee^{s\xi}]<\infty$ for some $s>0$. Then
\begin{equation}\label{thm: light deviation_s_claim}
\lim_{t\to\infty}\frac{-\log \mmp\{\hat N(t)=\lfloor b\mu^{-1}t\rfloor \}}{t^2}=\int_{b\mu^{-1}}^{\mu^{-1}}yI(1/y){\rm d}y<\infty,
\end{equation}
where $\mu=\me [\xi]<\infty$.

Under the additional assumptions that the distribution of $\xi$ is nonlattice and $\mu/b<A_0$,
\begin{equation}\label{thm: light deviation_d_claim}
-\log \mmp\{\hat N(t)=\lfloor b\mu^{-1}t\rfloor \}=t^2 \cdot \int_{b\mu^{-1}}^{\mu^{-1}}yI(1/y){\rm d}y + \frac{1-b}{2\mu}t \log t + O(t),\quad t\to \infty.
\end{equation}
\end{theorem}

In the `universal' case $b>1$, we shall need a counterpart of the function $I$ derived from the distribution of $-\xi$. In this case, without any assumptions on exponential moments of $\xi$, the Legendre transform $\tilde{I}$ of $s\mapsto \log\mathbb{E}\big[\eee^{-s\xi}\big]$ satisfies
$$
\tilde{I}(x)=\sup_{s\geq 0}\,(sx-\log \Lambda(-s)),\quad x\geq -\mu,
$$
see again Lemma 2.2.5 in~\cite{Dembo+Zeitouni:2009}. Put
$$
I^{\ast}(x)=\tilde{I}(-x)=\sup_{s\geq 0}\,(-sx-\log \Lambda(-s))=\sup_{s\leq 0}\,(sx-\log \Lambda(s)),\quad x\leq \mu.
$$

\begin{theorem}\label{thm: light deviation2}
Let $b>1$. Assume that $\mmp\{\xi\leq x\}<1$ for all $x\geq 0$, $\mmp\{\xi\leq b^{-1}\mu\}>0$ and $\mu =\me[\xi]<\infty$. Then
\begin{equation}\label{thm: light deviation2_s_claim}
\lim_{t\to\infty}\frac{-\log \mmp\{\hat N(t)=\lfloor b\mu^{-1}t\rfloor \}}{t^2}=\int_{\mu^{-1}}^{b\mu^{-1}}yI^\ast(1/y){\rm d}y<\infty.
\end{equation}
Under the additional assumptions that the distribution of $\xi$ is nonlattice,
\begin{equation}\label{thm: light deviation2_d_claim}
-\log \mmp\{\hat N(t)=\lfloor b\mu^{-1}t\rfloor \}=t^2 \cdot \int_{\mu^{-1}}^{b\mu^{-1}}yI^\ast(1/y){\rm d}y + \frac{b-1}{2\mu}t \log t + O(t), \quad t\to \infty.
\end{equation}
\end{theorem}

\begin{example}
Put $\varphi(s):=\eee^{-s^{1/2}}(1+s^{1/2})$ for $s\geq 0$. Since $-\varphi^\prime(s)=\eee^{-s^{1/2}}/2$, we infer that $\varphi$ is completely monotone. Hence, $\varphi(s)=\me [\eee^{-s\xi}]=\Lambda(-s)$ for some random variable $\xi>0$. Further, $\mu=\me [\xi]=-\varphi'(0)=1/2$, and an application of Theorem 8.1.6 in~\cite{Bingham+Goldie+Teugels:1989} yields $\mmp\{\xi>t\}\sim (6\pi^{1/2})^{-1}t^{-3/2}$ as $t\to\infty$. Thus, by Theorem~\ref{thm: regular deviation}, for $b\in(0,1)$, $$\lim_{t\to\infty}\frac{-\log \mmp\{\hat N(t)=\lfloor 2b t\rfloor \}}{t\log t}=1-b.$$ On the other hand, it can be checked that $I^\ast(x)=1/(4x)-x+\log (2x)$ for $0< x\leq 1/2$ and $$\int_{2}^{2b} yI^\ast(1/y){\rm d}y=\frac{2}{3}b^3-2b^2\log b+b^2-2b+\frac{1}{3},\quad b>1.$$ Hence, by Theorem \ref{thm: light deviation2}, for $b>1$,
$$-\log \mmp\{\hat N(t)=\lfloor 2 bt\rfloor\}=\Big(\frac{2}{3}b^3-2b^2\log b+b^2-2b+\frac{1}{3}\Big)t^2+(b-1)t\log t+O(t),\quad t\to\infty.$$
\end{example}
\begin{example}
Assume that $\xi$ has the gamma distribution with parameter $2/\rho$ and $1$, see formula~\eqref{eq:gamma}. Then $\mu=2/\rho$, $A_0=(2/\rho)\lim_{s\to 1-}(1-s)^{-1}=+\infty$, $I(x)=x-2/\rho+(2/\rho)\log (2/(\rho x))$ for $x\geq 2/\rho$ and $I^\ast$ is given by the same formula, but for $x\in (0,2/\rho)$. As a consequence, for $b\in (0,1)$, $$\int_{b\rho/2}^{\rho/2}yI(1/y){\rm d}y=\frac{\rho}{8}\big(2b^2\log (1/b)+(1-b)(1-3b)\big)$$ and, for $b>1$, $$\int_{\rho/2}^{b\rho/2}yI^\ast(1/y){\rm d}y=\frac{\rho}{8}\big(2b^2\log b-(b-1)(3b-1)\big).$$ Now Theorems~\ref{thm: light deviation} and \ref{thm: light deviation2} in combination with formula~\eqref{eq:principal} ensure that, for $b>0$, $b\neq 1$, $$-\log \mmp\{\Theta_\rho(D_t)=\lfloor b(\rho/2) t^\rho\rfloor\}=\frac{\rho}{8}\big|2b^2\log b-(b-1)(3b-1)\big|t^{2\rho}+\frac{|b-1|\rho^2}{4}t^\rho\log t+O(t^\rho)$$
as $t\to\infty$. In the case $\rho=2$, this extends Theorem 1.1 in~\cite{Shirai:2006}.
\end{example}

\subsection{Local limit theorem}

Recall the standing assumption that the distribution of $\xi$ is nondegenerate.
\begin{proposition}\label{thm:local limit}
The following asymptotic relations hold:
\begin{equation}\label{eq:diver}
\lim_{t\to\infty}\,{\rm Var}\,[\hat N(t)]=\infty
\end{equation}
and
\begin{equation}\label{eq:local}
\mmp\{\hat N(t)=\lfloor U(t)\rfloor\}~\sim~\frac{1}{(2\pi{\rm Var}[\hat N(t)])^{1/2}},\quad t\to\infty.
\end{equation}
\end{proposition}
\begin{remark}
Observe that $${\rm Var}[\hat N(t)]=\sum_{n\geq 1}\mmp\{S_n\leq t\}\mmp\{S_n>t\}\leq U(t)=O(t),\quad t\to\infty,$$ where the last equality is a consequence of the elementary renewal theorem. This together with Proposition \ref{thm:local limit} justifies the claim made in Remark \ref{rem:infinite}: under the assumptions of Theorem \ref{thm:infinite deviation}, $$\lim_{t\to\infty}\frac{-\log \mmp\{\hat N(t)=\lfloor U(t)\rfloor\}}{U(t)}=0.$$
\end{remark}
\begin{remark}
Under the assumption that the distribution of $\xi$ belongs to the domain of attraction of an $\alpha$-stable distribution for $\alpha\in (1,2]$ the asymptotic behavior of ${\rm Var}\,[\hat N(t)]$ as $t\to\infty$ was found in Corollary 5.3 of \cite{Alsmeyer+Iksanov+Kabluchko:2025}. For instance, if $\sigma^2:={\rm Var}\,[\xi]\in (0,\infty)$, then $${\rm Var}\,[\hat N(t)]~\sim~\Big(\frac{\sigma^2 t}{\mu^3\pi}\Big)^{1/2},\quad t\to\infty,$$ where $\mu=\me [\xi]<\infty$. Suppose \eqref{eq:domain alpha} with $\alpha\in [0,1)$. One can show along similar lines that $${\rm Var}\,[\hat N(t)]~\sim~\frac{c_\alpha}{\mmp\{\xi>t\}}~\sim~\frac{c_\alpha t^\alpha}{\ell(t)},\quad t\to\infty,$$ where $c_\alpha:=\int_0^\infty\mmp\{W_\alpha^\leftarrow(1)>y\}\mmp\{W_\alpha^\leftarrow(1)\leq y\}{\rm d}y<\infty$ if $\alpha\in (0,1)$ and $c_0:=\int_0^\infty\mmp\{Z_0>y\}\mmp\{Z_0\leq y\}{\rm d}y=1/2$.

We stress that Proposition \ref{thm:local limit} holds under the sole (standing) assumption that the distribution of $\xi$ is nondegenerate. In this more general setting, the divergence of ${\rm Var}\,[\hat N(t)]$ requires a proof, which cannot appeal to any results available under the domain of attraction assumption.
\end{remark}

\section{Proofs}

\subsection{Proof of Theorem \ref{thm:infinite deviation}}

Relation \eqref{eq:ren asymp} is known. It follows by an application of Karamata's Tauberian theorem, see, for instance, formulae (8.6.1) and (8.6.3) on p.~361 in \cite{Bingham+Goldie+Teugels:1989}.

For each $n\in\mn$ and each $t>0$, put $\eta_n(t):=\1_{\{\hat S_n\leq t\}}$. For each $s\neq 0$ and each $t>0$, define a new probability measure $\mmp^{(s,\,t)}$ on the $\sigma$-algebra generated by $\hat S_1$, $\hat S_2,\ldots$ as follows:
\begin{equation}\label{eq:change of measure}
\me^{(s,\,t)}[g(\eta_1(t),\ldots, \eta_k(t))]:=\frac{\me \big[\eee^{s\hat N(t)}g(\eta_1(t),\ldots, \eta_k(t))\big]}{\me [\eee^{s\hat N(t)}]}
\end{equation}
for all $k\in\mn$ and all bounded measurable functions $g:\mr^k\to\mr$, where $\me^{(s,\,t)}$ is the expectation with respect to $\mmp^{(s,\,t)}$. For any bounded measurable function $h:\mr\to\mr$, put $g(x_1,\ldots, x_k):=h(x_1+\ldots+x_k)$ for $(x_1,\ldots, x_k)\in\mr^k$. Using \eqref{eq:change of measure} with the so defined $g$ and then sending $k\to\infty$ we arrive at
\begin{equation}\label{eq:change of measure2}
\me^{(s,\,t)}[h(\hat N(t))]=\frac{\me \big[\eee^{s\hat N(t)}h(\hat N(t))\big]}{\me [\eee^{s\hat N(t)}]}.
\end{equation}
By an approximation argument, equality \eqref{eq:change of measure2} holds for not necessarily bounded Borel functions $h:\mr\to\mr$ as long as the left- or right-hand side of the equality is well defined, possibly infinite. We stress that, thanks to the independence of $\hat S_1$, $\hat S_2,\ldots$, 
the random variables $\eta_1(t)$, $\eta_2(t),\ldots$ are still independent under $\mmp^{(s,\,t)}$. A specialization of \eqref{eq:change of measure2}, with $h(x) = \eee^{-sx} \1_{\{\lfloor b\,U(t)\rfloor\}}(x)$ (indicator function of the one-point set), yields
\begin{equation}\label{eq:change of measure3}
\mmp\{\hat N(t)=\lfloor b\,U(t)\rfloor\}=\eee^{-s\lfloor b\,U(t)\rfloor}
\me[\eee^{s\hat N(t)}]\mmp^{(s,\,t)}\{\hat N(t)=\lfloor b\,U(t)\rfloor\}.
\end{equation}

For each $n\in\mn$ and each $t>0$, put $p_n=p_n(t):=\mmp\{S_n\leq t\}$. Then, for $s\in\mr$, $$\psi_t(s):=\log \me \big[\eee^{s\hat N(t)}\big]=\sum_{n\geq 1}\log ((\eee^s-1)p_n+1) 
$$ and further $$\psi^\prime_t(s)=\sum_{n\geq 1}\frac{\eee^s p_n}{(\eee^s-1)p_n+1},$$ $$\psi^{\prime\prime}_t(s)=\sum_{n\geq 1}\frac{\eee^s p_n(1-p_n)}{((\eee^s-1)p_n+1)^2}=\sum_{n\geq 1}\frac{\eee^s p_n}{(\eee^s-1)p_n+1}\Big(1-\frac{\eee^s p_n}{(\eee^s-1)p_n+1}\Big),\quad s\in\mr$$ and $$\psi^{\prime\prime\prime}_t(s)=\sum_{n\geq 1}\frac{\eee^s p_n}{(\eee^s-1)p_n+1}\Big(1-\frac{\eee^s p_n}{(\eee^s-1)p_n+1}\Big)\Big(1-\frac{2\eee^s p_n}{(\eee^s-1)p_n+1}\Big).
$$ 
For later use, we note that
\begin{equation}\label{eq:ineq}
\big|\psi^{\prime\prime\prime}_t(s)\big|\leq \sum_{n\geq 1}\frac{\eee^s p_n}{(\eee^s-1)p_n+1}=\psi^\prime_t(s),\quad s\in\mr.
\end{equation}

Observe that $$f_\alpha^\prime(s)=\int_0^\infty \frac{\eee^s\mmp\{Z_\alpha>y\}}{(\eee^s-1)\mmp\{Z_\alpha>y\}+1}{\rm d}y,\quad s\in\mr.$$ We already know that $f_\alpha$ is strictly convex on $\mr$. Hence, $f^\prime_\alpha$ is strictly increasing on $\mr$. Also, it can be seen, with the help of Lebesgue's dominated convergence theorem, that $f_\alpha^\prime$ is continuous on $\mr$. Furthermore, by Levi's monotone convergence theorem, $\lim_{s\to -\infty}f^\prime_\alpha(s)=0$ and $\lim_{s\to\infty}f_\alpha^\prime(s)=+\infty$. Hence, given $b>0$, the equation $f_\alpha^\prime(s)=b$ has a unique solution that we denote by $s_b$. Since $f_\alpha^\prime(0)=\me[Z_\alpha]=1$, we conclude that $s_b>0$
if $b>1$ and $s_b<0$ if $b\in (0,1)$. For each $b>0$, $J_\alpha(b)=bs_b-f_\alpha(s_b)$. In particular, $J_\alpha(1)=0$ and $J_\alpha^\prime(b)=s_b$ for $b>0$, whence $J_\alpha(b)>0$ whenever $b>0$, $b\neq 1$.

Next, we show that, for each fixed $s\in\mr$,
\begin{equation}\label{eq:basic convergence}
\lim_{t\to\infty}\frac{\psi_t(s)}{U(t)}=f_\alpha(s)\quad\text{and}\quad \lim_{t\to\infty}\frac{\psi^\prime_t(s)}{U(t)}=f_\alpha^\prime(s).
\end{equation}
Put $\nu(t):=\inf\{k\geq 1: S_k>t\}$ for $t\geq 0$. It is known (see, for instance, Theorem 7 in \cite{Feller:1949} in the case $\alpha\in (0,1)$ and Corollary 2.3 in \cite{Kabluchko+Marynych:2016} in the case $\alpha=0$) that $\nu(t)/U(t)$ converges in distribution to a random variable $Z_\alpha$ as $t\to\infty$. Put $S_0:=0$ and write, with the help of the equality $\{S_k\leq t\}=\{\nu(t)>k\}$, which holds for $k\in\mn_0$ and $t\geq 0$, and the change of variable,
$$
\psi^\prime_t(s)+1=\int_0^\infty\frac{\eee^s\mmp\{S_{\lfloor x\rfloor}\leq t\}}{(\eee^s-1)\mmp\{S_{\lfloor x\rfloor}\leq t\}+1}{\rm d}x=U(t)\int_0^\infty\frac{\eee^s \mmp\{\nu(t)>\lfloor yU(t)\rfloor\}}{(\eee^s-1)\mmp\{\nu(t)>\lfloor yU(t)\rfloor\}+1}{\rm d}y.
$$
Since the function $y\mapsto\mmp\{Z_\alpha\leq y\}$ is continuous on $[0,\infty)$, we infer $$\lim_{t\to\infty}\mmp\{\nu(t)>\lfloor yU(t)\rfloor\}=\mmp\{Z_\alpha>y\},\quad y\geq0.$$ In fact, this convergence is uniform in $y$, but we do not need this uniformity. 
Further, $\int_0^\infty \mmp\{\nu(t)>\lfloor yU(t)\rfloor\}{\rm d}y=\sum_{k\geq 0}\mmp\{\nu(t)>k\}/U(t)=(U(t)+1)/U(t)\to 1$ as $t\to\infty$ and $\int_0^\infty \mmp\{Z_\alpha>y\}{\rm d}y=\me [Z_\alpha]=1$. Hence, by Scheffe's lemma, $$\lim_{t\to\infty}\int_0^\infty\Big|\mmp\{\nu(t)>\lfloor yU(t)\rfloor\}-\mmp\{Z_\alpha>y\}\Big|{\rm d}y=0.$$ For $x,y\in[0,1]$ and $s\in\mr$, $$\Big|\frac{\eee^s x}{(\eee^s-1)x+1}-\frac{\eee^s y}{(\eee^s-1)y+1}\Big|\leq \eee^s \max (1, \eee^{-2s})|x-y|=\eee^{|s|}|x-y|.$$ Using this inequality, together with the preceding centered formula, we obtain
$$\lim_{t\to\infty}\int_0^\infty \frac{\eee^s \mmp\{\nu(t)>\lfloor yU(t)\rfloor\}}{(\eee^s-1)\mmp\{\nu(t)>\lfloor yU(t)\rfloor\}+1}{\rm d}y=\int_0^\infty\frac{\eee^s \mmp\{Z_\alpha>y\}}{(\eee^s-1)\mmp\{Z_\alpha>y\}+1}{\rm d}y.$$ Thus, we have proved the second limit relation in \eqref{eq:basic convergence}, for each fixed $s\in\mr$. The convergence is actually locally uniform, as the convergence of monotone functions to a continuous limit. This entails the first limit relation in \eqref{eq:basic convergence}.

For later needs, we note that essentially the same argument enables us to conclude that, for each fixed $s\in\mr$,
\begin{equation}\label{eq:basic convergence2}
\lim_{t\to\infty}\frac{\psi_t^{\prime\prime}(s)}{U(t)}=f_\alpha^{\prime\prime}(s).
\end{equation}
Now we prove that the convergence in \eqref{eq:basic convergence2} is locally uniform in $s$. As a consequence of \eqref{eq:ineq}, for any finite $a_1<a_2$, $$\sup_{s\in [a_1,a_2]}\frac{|\psi^{\prime\prime\prime}_t(s)|}{U(t)}\leq \frac{\psi^\prime_t(a_2)}{U(t)}=O(1),\quad t\to\infty$$ by \eqref{eq:basic convergence}. Hence, for some $t_0>0$, the family of functions $(\psi^{\prime\prime}_t/U(t))_{t\geq t_0}$ is equicontinuous on $[a_1, a_2]$. Their pointwise convergence in \eqref{eq:basic convergence2}, together with the continuity of $f^{\prime\prime}_\alpha$, implies uniform convergence on $[a_1, a_2]$.

Using \eqref{eq:basic convergence} and $\lim_{t\to\infty}\lfloor bU(t)\rfloor/U(t)=b$, in combination with the continuity and strict monotonicity
of $\psi^\prime_t$, one can check that, for a fixed $b>0$ and all sufficiently large $t$, there exists a unique $s^{(t)}\in\mr$ satisfying $\psi^\prime_t\big(s^{(t)}\big)=\lfloor bU(t)\rfloor$. The local uniform convergence in \eqref{eq:basic convergence} and the strict monotonicity of $f^\prime_\alpha$ imply that $\lim_{t\to\infty}s^{(t)}=s_b$ and thereupon
$$\lim_{t\to\infty}\frac{s^{(t)}\lfloor bU(t)\rfloor-\psi_t\big(s^{(t)}\big)}{U(t)}=bs_b-f_\alpha(s_b)=J_\alpha(b).$$

Write $\me^{(s^{(t)},t)}[\hat N(t)]=\psi^\prime_t\big(s^{(t)}\big)=\lfloor bU(t)\rfloor$ having utilized \eqref{eq:change of measure2} with $h(x)=x$. Invoking once again \eqref{eq:change of measure2}, together with the locally uniform convergence in \eqref{eq:basic convergence2}, we obtain
$${\rm Var}^{(s^{(t)},t)}[\hat N(t)]=\psi^{\prime\prime}_t\big(s^{(t)}\big)~\sim~ f^{\prime\prime}_\alpha (s_b)U(t),\quad t\to\infty.$$ Observe that
$$f^{\prime\prime}_\alpha(s)=\int_0^\infty\frac{\eee^s\mmp\{Z_\alpha>y\}\mmp\{Z_\alpha\leq y\}}{\big(1+(\eee^s-1)\mmp\{Z_\alpha>y\}\big)^2}{\rm d}y>0,\quad s\in\mr.$$ In particular, $f^{\prime\prime}_\alpha(s_b)>0$ and thereupon $\lim_{t\to\infty}{\rm Var}^{(s^{(t)},t)}[\hat N(t)]=\infty$. With this at hand, an application of Proposition \ref{prop:Dolg} yields
$$-\log \mmp^{(s^{(t)},\,t)}\{\hat N(t)=\lfloor b\,U(t)\rfloor\}=2^{-1}\log (2\pi {\rm Var}^{(s^{(t)},\,t)}[\hat N(t)])+o(1)=O(\log U(t))=o(U(t)),\quad t\to\infty.$$
Using \eqref{eq:change of measure3} with $s=s^{(t)}$, we conclude that $$-\log\mmp\{\hat N(t)=\lfloor bU(t)\rfloor\}=s^{(t)}\lfloor bU(t)\rfloor-\psi_t\big(s^{(t)}\big)+o(U(t))=J_\alpha(b)U(t)+o(U(t)),\quad t\to\infty.$$
The proof of Theorem \ref{thm:infinite deviation} is complete.

\subsection{Proofs of Theorems \ref{thm: regular deviation}, \ref{thm:semi}, \ref{thm: light deviation} and \ref{thm: light deviation2}}

Following the argument given at the beginning of the proof of Theorem 1.1 in \cite{Shirai:2006} we first show that as far as the logarithmic asymptotic behavior of $\mmp\{\hat N(t)=\lfloor b\mu^{-1}t\rfloor\}$ is concerned one can assume that the sequence $(\1_{\{\hat S_n\leq t\}})_{n\geq 1}$ has an initial block of $1$s of size $\lfloor b\mu^{-1}t\rfloor$ followed by $0$s. We stress that such a simplification is not possible under the assumptions of Theorem \ref{thm:infinite deviation}. The reason is that $\lim_{t\to\infty}\mmp\{S_{\lfloor v\,U(t)\rfloor}\leq t\}\in (0,1)$ for each $v>0$, whereas in the case $\mu<\infty$, by the weak law of large numbers, $\lim_{t\to\infty}\mmp\{S_{\lfloor v\mu^{-1}t\rfloor}\leq t\}$ is either $1$ or $0$ depending on whether $v\in (0,1)$ or $v>1$.
\begin{lemma}\label{lem:shirai}
Let $b>0$, $b\neq 1$. Assume that $\mu=\me [\xi]<\infty$, $\mmp\{\xi\leq x\}<1$ for all $x\geq 0$ and, if $b>1$, that $\mmp\{\xi\leq b^{-1}\mu\}>0$.  Then
$$\log \mmp\{\hat N(t)=\lfloor b\mu^{-1}t\rfloor\}= \log \bigg(\prod_{n=1}^{\lfloor b\mu^{-1}t\rfloor}\mmp\{S_n\leq t\}\prod_{j\geq \lfloor b\mu^{-1}t\rfloor+1}\mmp\{S_j> t\}\bigg) +O(t),\quad t\to\infty.$$
\end{lemma}
\begin{remark}
The assumptions concerning the support of the distribution of $\xi$ ensure that the logarithm on the right-hand side is not equal to $-\infty$ for large fixed $t$. Although a justification of this claim is scattered over the proof of Lemma \ref{lem:shirai}, we prefer to give here a precise argument. For $n\in \{1,2,\ldots,\lfloor b\mu^{-1}t\rfloor\}$,
\begin{multline}\label{eq:shirai_not_vanishing}
\mmp\{S_n\leq t\}\geq \mmp\{S_{\lfloor b\mu^{-1}t\rfloor}\leq t\}\geq \mmp\{\xi_1\leq t/\lfloor b\mu^{-1}t\rfloor,\ldots, \xi_{\lfloor b\mu^{-1}t\rfloor}\leq t/\lfloor b\mu^{-1}t\rfloor\}\\\geq (\mmp\{\xi\leq b^{-1}\mu\})^{\lfloor b\mu^{-1}t\rfloor},\quad t\geq b^{-1}\mu.
\end{multline}
If $b\in (0,1)$, then $\mmp\{\xi\leq b^{-1}\mu\}>0$ holds automatically, whereas, if $b>1$, it holds by
assumption. The assumption $\mmp\{\xi\leq x\}<1$ for all $x\geq 0$ ensures $\mmp\{S_j>t\}>0$ for all $j\geq 1$ and $t>0$. Thus the product under the logarithm does not vanish for 
large fixed $t$.
\end{remark}
\begin{proof}[Proof of Lemma~\ref{lem:shirai}]
Put $m:=\lfloor b\mu^{-1}t\rfloor$. In what follows $|A|$ denotes the number of elements of a finite set $A$ and $A^c$ denotes the complement of $A$. Recall the notation $p_n(t)=\mmp\{S_n\leq t\}$ for $n\in\mn$ and $t>0$. Write
\begin{multline*}
\mmp\{\hat N(t)=m\}=\mmp\{\hat S_n\leq t~\text{with}~n\in A_t~\text{for some}~A_t\subset \mn,|A_t|=m~\text{and}~\hat S_n>t~\text{with}~n\in A_t^c\}\\=\sum_{A_t: |A_t|=m}\prod_{n\in A_t}p_n(t)\prod_{j\in A_t^c}(1-p_j(t))\geq \prod_{n=1}^m p_n(t)\prod_{j\geq m+1}(1-p_j(t)).
\end{multline*}
Thus, it remains to show that, as $t\to\infty$,
$$ \log \Big(\sum_{A_t: |A_t|=m}\prod_{n\in A_t}p_n(t)\prod_{j\in A_t^c}(1-p_j(t))\Big) \leq \log \Big(\prod_{n=1}^m p_n(t)\prod_{j\geq m+1}(1-p_j(t))\Big) + O(t).$$
Put $q(x):=x(1-x)^{-1}$ for $x\in [0,1)$. The function $q$ strictly increases and satisfies $q(x)\leq 3x$ for $x\in [0,2/3]$. Observe that $1-p_j(t)>0$ for each $j\in\mn$ and each $t>0$ by assumption, that $-\log (1-p_j(t))\sim p_j(t)$ as $j\to\infty$ and $\sum_{j\geq 1}p_j(t)=U(t)<\infty$ for each $t>0$. Hence, $\prod_{j\geq 1}(1-p_j(t))>0$ for each $t>0$. Writing $$\sum_{A_t: |A_t|=m}\prod_{n\in A_t}p_n(t)\prod_{j\in A_t^c}(1-p_j(t))=\prod_{j\geq 1}(1-p_j(t))\sum_{A_t: |A_t|=m}\prod_{n\in A_t}q(p_n(t))$$ and $$\prod_{n=1}^m p_n(t)\prod_{j\geq m+1}(1-p_j(t))=\prod_{j\geq 1}(1-p_j(t))\prod_{n=1}^m q(p_n(t))$$ and dividing both equalities by $\prod_{j\geq 1}(1-p_j(t))$ we conclude that it is enough to prove that
\begin{equation}\label{eq:approx}
\log \Big(\sum_{A_t: |A_t|=m}\prod_{n\in A_t}q(p_n(t))\Big)  \leq \log \Big(\prod_{n=1}^m q(p_n(t))\Big) + O(t),\quad t\to \infty.
\end{equation}

Now we have to treat two cases separately.

\noindent {\sc Case $b>1$}. Each set $A_t$ containing $m$ positive integers coincides with some set $B_{t,\,n} \subseteq \{1,2,\ldots,n\}$ such that $n\in B_{t,\,n}$ for some $n\geq m$. In other words, $n$ is the largest element of $A_t=B_{t,n}$. Decomposing with respect to this largest element gives a representation $$\sum_{A_t: |A_t|=m}\prod_{n\in A_t}q(p_n(t))=\sum_{n\geq m}\sum_{B_{t,\,n}}\prod_{j\in B_{t,\,n}}q(p_j(t)).$$ Each $m$-tuple coming from $B_{t,\,n}$ can be written (omitting the argument $t$ for simplicity) as $q(p_{j_1})\cdot\ldots\cdot q(p_{j_{m-1}})q(p_n)$ for some positive integers $j_1<j_2<\ldots<j_{m-1}\leq n-1$. The number of such tuples is $\binom{n-1}{m-1}$. Since $q$ is an increasing function, and $k\mapsto p_k$ is nonincreasing, we infer $q(p_{j_1})\leq q(p_1),\ldots,q(p_{j_{m-1}})\leq q(p_{m-1})$. This yields $$\sum_{A_t: |A_t|=m}\prod_{n\in A_t}q(p_n(t))\leq \prod_{j=1}^{m-1} 
q(p_j(t))\sum_{n\geq m}\binom{n-1}{m-1}q(p_n(t)).$$ For $n\geq m$, $p_n(t)\leq p_m(t)$, and $p_m(t)$ is the probability of a large deviation event $\{S_{\lfloor b\mu^{-1}t\rfloor}\leq t\}$. In particular, 
$b>1$ secures $\lim_{t\to\infty}p_m(t)=0$. Hence, for large $t$ and all $n\geq m$, $p_n(t)\leq 2/3$, whence $$\prod_{j=1}^{m-1} q(p_j(t))\sum_{n\geq m}\binom{n-1}{m-1}q(p_n(t))\leq 3 \prod_{j=1}^{m-1} q(p_j(t))\sum_{n\geq m}\binom{n-1}{m-1}p_n(t)$$ for large $t$. There are several ways to prove the equality $$\sum_{n\geq m}\binom{n-1}{m-1}p_n(t)=U^{\ast(m)}(t),\quad t>0,$$ where $U^{\ast (m)}$ is the $m$-fold Lebesgue-Stieltjes convolution of $U$ with itself. Perhaps, the simplest way is to use Laplace-Stieltjes transforms. Indeed, put $\varphi(s):=\me [\eee^{-s\xi}]$ for $s\geq 0$. Then $\int_{[0,\,\infty)}\eee^{-st}{\rm d}U(t)=\varphi(s)(1-\varphi(s))^{-1}$ for $s>0$ and thereupon $$\int_{[0,\,\infty)}\eee^{-st}{\rm d}U^{\ast (m)}(t)=\Big(\frac{\varphi(s)}{1-\varphi(s)}\Big)^m,\quad s>0.$$ On the other hand, using $$\sum_{n\geq m}\binom{n-1}{m-1}x^n=\Big(\frac x{1-x}\Big)^m,\quad x\in [0,1)$$
we infer $$\int_{[0,\,\infty)}\eee^{-st}{\rm d}\Big(\sum_{n\geq m} \binom{n-1}{m-1}\mmp\{S_n\leq t\}\Big)=\sum_{n\geq m}\binom{n-1}{m-1}(\varphi(s))^n=\Big(\frac{\varphi(s)}{1-\varphi(s)}\Big)^m,\quad s>0.$$ Summarizing, we have already shown 
that $$\sum_{A_t: |A_t|=m}\prod_{n\in A_t}q(p_n(t))\leq 3\prod_{j=1}^{m-1}q(p_j(t)) U^{\ast(m)}(t)$$ for large $t$.

To prove that
\begin{equation}\label{eq:convolution}
\log U^{\ast (m)}(t) = O(t) \quad \mbox{as } t\to \infty,
\end{equation}
we use a crude estimate $U^{\ast (m)}(t)\leq \sum_{k\geq 1}U^{\ast (k)}(t)$ for large $t$. Since $$\int_{[0,\infty)}\eee^{-st}{\rm d}\Big(\sum_{k\geq 1}U^{\ast (k)}(t)\Big)=\sum_{k\geq 1}\Big(\frac{\varphi(s)}{1-\varphi(s)}\Big)^k=\frac{\varphi(s)}{1-2\varphi(s)}<\infty,\quad s>s_0,$$ where $s_0$ is the unique solution to $\varphi(s)=1/2$, we conclude that $\sum_{k\geq 1}U^{\ast (k)}(t)=o(\eee^{(s_0+\varepsilon)t})$ as $t\to\infty$ for all $\varepsilon>0$, by Theorem 2.2a on p.~39 in \cite{Widder:1946}. This entails\footnote{Assume that the distribution of $\xi$ is nonlattice. Then, by Corollary 3.4 in \cite{Buraczewski+Iksanov+Marynych:2025} with $\alpha= b^{-1}\mu$, $j=\lfloor b\mu^{-1}t\rfloor$ and bounded $y$, $$U^{\ast (m)}(t)~\sim~\frac{\mu^{1/2}}{\theta (2\pi b\sigma^2)^{1/2}}\frac{\eee^{\theta t}}{t^{1/2}}\Big(\frac{\varphi(\theta)}{1-\varphi(\theta)}\Big)^m,\quad t\to\infty$$ for appropriate positive constants $\theta$ and $\sigma^2$. Although, as our argument above demonstrates, such a precision is not needed, the latter limit relation also secures \eqref{eq:convolution}.} \eqref{eq:convolution}. As a consequence, $$ \log \Big(\sum_{A_t: |A_t|=m}\prod_{n\in A_t}q(p_n(t))\Big) \leq \log \prod_{j=1}^{m-1}q(p_j(t)) + O(t), \quad t\to \infty.$$ To replace $\prod_{j=1}^{m-1}$ with $\prod_{j=1}^m$ on the right-hand side, observe that~\eqref{eq:shirai_not_vanishing} implies that
$$
\log p_m(t)\geq m\log F(\mu/b),\quad t\geq b^{-1}\mu,
$$
where $F$ is the distribution function of $\xi$. Recalling that $\lim_{t\to\infty}\log (1-p_m(t))=0$ we finally arrive at $\log q(p_m(t))=\log p_m(t)-\log (1-p_m(t))=O(t)$ as $t\to\infty$. The proof of~\eqref{eq:approx} is complete in the case $b>1$.

\noindent {\sc Case $b\in (0,1)$}. Put $n:=\lfloor c\mu^{-1}t\rfloor$ for $c>3b\vee 1$ and define, for $k\in\{0,1,\ldots, m\}$, the class of sets $$\mathcal{A}_t(m,k):=\{A_t\subset \mn: |A_t|=m, |A_t \cap \{n+1,n+2,\ldots\}|=k\}.$$ Then $$\sum_{A_t: |A_t|=m}\prod_{n\in A_t}q(p_n(t))=\sum_{k=0}^m\sum_{A_t\in \mathcal{A}_t(m,k)}\prod_{i\in A_t\cap \{1,2,\ldots, n\}}q(p_i(t))\prod_{j\in A_t\cap\{n+1,n+2,\ldots\}}q(p_j(t)).$$ The first products contain $m-k$ factors, the number of such products is $\binom{n}{m-k}$, and these are of the form (omitting again $t$ for simplicity) $q(p_{i_1})\cdot\ldots\cdot q(p_{i_{m-k}})$ with positive integer $i_1<\ldots<i_{m-k}\leq n$. Since $q$ is an increasing function and $k\mapsto p_k$ is decreasing, we infer $q(p_{i_1})\leq q(p_1),\ldots$, $q(p_{i_{m-k}})\leq q(p_{m-k})$. The second products are of the form $q(p_{i_{m-k+1}})\cdot\ldots\cdot q(p_{i_m})$, and each such a product is dominated by $(k!)^{-1}(\sum_{j\geq n+1}q(p_j(t)))^k$. This is justified by the multinomial theorem, which states that the coefficient in front of the product of monomials is $k!$. Summarizing, $$\sum_{A_t: |A_t|=m}\prod_{n\in A_t}q(p_n(t))\leq \sum_{k=0}^m\binom{n}{m-k} \prod_{i=1}^{m-k}q(p_i(t))\frac{1}{k!}\Big(\sum_{j\geq n+1}q(p_j(t))\Big)^k.$$ The sequence $i \mapsto \binom{n}{i}$ is increasing for $i\leq \lfloor n/2\rfloor$. Since $m-k\leq m\leq \lfloor n/2\rfloor$ by the choice of $n$, we conclude that $\binom{n}{m-k}\leq \binom{n}{m}$. Since $p_i(t)\geq p_m(t)$ for $i\in\{1,\ldots, m\}$ and $\lim_{t\to\infty}p_m(t)=1$, in particular, $p_m(t)\geq 1/2$ for large $t$, we infer $q(p_i(t))\geq q(p_m(t))\geq q(1/2)=1$. Hence, $\prod_{i=1}^{m-k}q(p_i(t))\leq \prod_{i=1}^m q(p_i(t))$ for large $t$. As a consequence, $$\sum_{A_t: |A_t|=m}\prod_{n\in A_t}q(p_n(t))\leq \binom{n}{m}\prod_{i=1}^m q(p_i(t))\sum_{k\geq 0}\frac{1}{k!}\Big(\sum_{j\geq n+1}q(p_j(t))\Big)^k$$ for large $t$. For $j\geq n+1$ and large $t$, $p_j(t)\leq 2/3$. Also, $\binom{n}{m}\leq \sum_{k=0}^n\binom{n}{k}=2^n$. Hence, $$\binom{n}{m}\sum_{k\geq 0}\frac{1}{k!}\Big(\sum_{j\geq n+1}q(p_j(t))\Big)^k\leq 2^n \sum_{k\geq 0}\frac{1}{k!}\Big(3\sum_{j\geq n+1}p_j(t)\Big)^k\leq 2^n\eee^{3U(t)}.$$ By the elementary renewal theorem, $$ \log\Big(\binom{n}{m}\sum_{k\geq 0}\frac{1}{k!}\Big(\sum_{j\geq n+1}q(p_j(t))\Big)^k\Big) = O(t),\quad t\to\infty.$$ This completes the proof in the case $b\in(0,1)$.
\end{proof}

\begin{lemma}\label{lem:appr}
Suppose $\mu=\me [\xi]<\infty$. Then, as $t\to\infty$,
\begin{equation}\label{eq:appr2}
-\log \prod_{n=1}^{\lfloor \mu^{-1} t\rfloor}\mmp\{S_n\leq t\}\ =\ O(t)
\end{equation}
and, under the additional assumption $\liminf_{t\to\infty}\mmp\{S_{\lfloor t/\mu\rfloor}>t\}>0$,
\begin{equation}\label{eq:appr1}
-\log \prod_{n\ge \lfloor \mu^{-1} t\rfloor+1}\mmp\{S_n>t\}\ =\ O(t).
\end{equation}
\end{lemma}
\begin{proof}
We first prove~\eqref{eq:appr2}. Recall that it is our standing assumption that the distribution of $\xi$ is nondegenerate. Observe that
$$
-\frac{1}{t}\log \prod_{n=1}^{\lfloor \mu^{-1} t\rfloor}\mmp\{S_n\leq t\}=\frac{1}{t}\sum_{n=1}^{\lfloor \mu^{-1} t\rfloor}(-\log \mmp\{S_n\leq t\})\leq -\frac{1}{\mu}\log\mathbb{P}\{S_{\lfloor \mu^{-1}t\rfloor+1}\leq t\}.
$$
Put $n:=\lfloor \mu^{-1}t\rfloor+1$ and observe that $n\leq \mu^{-1}t+1$ implies $t\geq \mu(n-1)$. Thus, it suffices to show that there exists $c\in (0,1]$ such that for all sufficiently large $n\in\mathbb{N}$,
\begin{equation}\label{eq:appr2_proof1}
\mathbb{P}\{S_n\leq \mu(n-1)\}\geq c.
\end{equation}
This relation is obvious if $\mathbb{E}[\xi^2]<\infty$ because, by the central limit theorem, $\lim_{n\to\infty}\mmp\{S_n\leq \mu(n-1)\}=1/2$. Thus, in what follows we assume that $\mathbb{E}[\xi^2]=\infty$. To prove~\eqref{eq:appr2_proof1} we apply the argument given at the beginning of the proof of Theorem 2.1
in~\cite{Jain+Pruitt:1987}, with $x_n=\mu(1-n^{-1})$. Define the function $m^{\ast}:[0,+\infty)\mapsto \mathbb{R}$ by $m^{\ast}(s)=\mathbb{E}[\xi\eee^{-s\xi}]/\mathbb{E}[\eee^{-s\xi}]=\Lambda'(-s)/\Lambda(-s) 
$ and observe that $m^{\ast}$ is continuous and strictly decreasing on $(0,+\infty)$ with $m^{\ast}(0+)=\mu$. Let $\lambda_n>0$ be the unique solution to $m^{\ast}(\lambda_n)=x_n$. Put also $R(s):=-\log \Lambda(-s)-sm^{\ast}(s)$ for $s\geq 0$. According to the aforementioned argument from~\cite{Jain+Pruitt:1987} it suffices to show that
\begin{equation}\label{eq:appr2_proof2}
\lim_{n\to\infty}nR(\lambda_n)=0.
\end{equation}
Using $m^{\ast}(0+)=\mu$ we conclude $m^{\ast}(0+)-m^{\ast}(\lambda_n)=\mu n^{-1}$. It is clear that $\lambda_n\to 0+$ as $n\to\infty$. Thus, by the mean value theorem for differentiable functions
$$
(m^{\ast})^{\prime}(\theta_n)\lambda_n=-\frac{\mu}{n}
$$
for some $\theta_n\to 0+$. From this and $\lim_{s\to 0+}(m^{\ast})^{\prime}(s)=-{\rm Var}\,\xi=-\infty$ we obtain that $\lambda_n=o(1/n)$ as $n\to\infty$. In particular, given $\varepsilon>0$ the inequality $\lambda_n\leq \varepsilon n^{-1}$ holds for all sufficiently large $n\in\mathbb{N}$. Note that $nR(\lambda_n)=-n\log\Lambda(-\lambda_n)-n\lambda_n x_n$ and $n\lambda_n x_n\to 0$ as $n\to\infty$. Furthermore, 
$$
-n\log\Lambda(-\lambda_n)\leq -n\log\Lambda(-\varepsilon/n)\sim n(1-\Lambda(-\varepsilon/n))~\to~\varepsilon\mu,\quad n\to\infty.
$$
Here, the first inequality holds for all sufficiently large $n\in\mathbb{N}$. Since $\varepsilon>0$ is arbitrary, this finishes the proof of~\eqref{eq:appr2_proof2} and hence~\eqref{eq:appr2_proof1}. 

Relation \eqref{eq:appr1} was proved in Lemma 6.1 of \cite{Alsmeyer+Iksanov+Kabluchko:2025}, see also pp.~123-124 in \cite{Hough+Krishnapur+Peres+Virag:2009}, under the additional assumption that the distribution of $\xi$ belongs to the domain of attraction of an $\alpha$-stable distribution for some $\alpha\in (1,2]$. Actually, the proof of Lemma 6.1 of \cite{Alsmeyer+Iksanov+Kabluchko:2025} only uses $\liminf_{t\to\infty}\mmp\{S_{\lfloor t/\mu\rfloor}>t\}>0$, which is a consequence of the domain-of-attraction assumption. The proof of Lemma \ref{lem:appr} is complete. 
\end{proof}
\begin{remark}
Although relations \eqref{eq:appr2} and \eqref{eq:appr1} may look symmetric, the underlying estimates are not. The source of this asymmetry is the nonnegativity of $\xi$. Indeed, $\liminf_{t\to\infty}
\mmp\{S_{\lfloor t/\mu\rfloor}>t\}>0$ does not hold in general and has to be imposed as an additional assumption, whereas its counterpart $\liminf_{t\to\infty}\mmp\{S_{\lfloor t/\mu\rfloor}\leq t\}>0$ holds automatically whenever $\xi$ is nonnegative with a nondegenerate distribution of finite mean.

The additional assumption for \eqref{eq:appr1} is satisfied under the assumptions of Theorems~\ref{thm: regular deviation}, \ref{thm:semi} and \ref{thm: light deviation}. Indeed, under \eqref{eq:domain alpha} with $\alpha>1$, the distribution of $\xi$ belongs to the domain of attraction of a stable distribution of index $\min (\alpha, 2)$, whereas the assumptions of Theorems~\ref{thm:semi} and \ref{thm: light deviation} imply that
$\me[\xi^2]<\infty$.
\end{remark}

According to Lemma \ref{lem:shirai} and relation \eqref{eq:appr1}, when proving Theorems \ref{thm: regular deviation}, \ref{thm:semi} and formula \eqref{thm: light deviation_s_claim} in Theorem~\ref{thm: light deviation} it suffices to show that the limit relations in focus hold true with \newline $-\log \prod_{n=\lfloor b \mu^{-1} t\rfloor+1}^{\lfloor \mu^{-1}t\rfloor}\mmp\{S_n>t\}$ replacing $-\log \mmp\{\hat N(t)=\lfloor b\mu^{-1}t\rfloor \}$. This can be done along the lines of the proof of Theorems 3.5(b), 3.6 and 3.1(a) in \cite{Alsmeyer+Iksanov+Kabluchko:2025}, in which the asymptotic behavior of $-\log \mmp\{\hat N(t)=0\}=-\log \prod_{n\geq 1}\mmp\{S_n>t\}$ has been investigated. In view of this we only provide sketches referring to~\cite{Alsmeyer+Iksanov+Kabluchko:2025} for details. We give a full proof of formula \eqref{thm: light deviation_d_claim} because its counterpart was not obtained in \cite{Alsmeyer+Iksanov+Kabluchko:2025}.

\begin{proof}[Proof of Theorem \ref{thm: regular deviation}]
We intend to prove that $$\lim_{t\to\infty}\frac{-\log \prod_{n=\lfloor b\mu^{-1}t\rfloor}^{\lfloor \mu^{-1} t\rfloor}\mmp\{S_n> t\}}{t\log t}\ =\ \frac{(\alpha-1)(1-b)}{\mu}.$$ \

By Theorem 3.3 in \cite{Cline+Hsing:2023}, for all $\delta>0$,
\begin{equation}\label{eq:Cline+Hsing LDP}
\lim_{n\to\infty}\sup_{t\ge \delta n}\Big|\frac{\mmp\{S_n- \mu n>t\}}{n\mmp\{\xi>t\}}-1\Big|\ =\ 0.
\end{equation}
Thus, for any $\delta\in (0,\mu b^{-1}(1-b))$,
$$ \sum_{n=\lfloor b\mu^{-1}t\rfloor}^{\lfloor (\mu+\delta)^{-1}t\rfloor}\log \mmp\{S_{n}-\mu n>t-\mu n\}~\sim~ \sum_{n=\lfloor b\mu^{-1}t\rfloor}^{\lfloor (\mu+\delta)^{-1}t \rfloor}\big(\log n+\log \mmp\{\xi>t-\mu n\}\big),\quad t\to\infty.$$
Observe that $$\sum_{n=\lfloor b\mu^{-1}t\rfloor}^{\lfloor (\mu+\delta)^{-1}t\rfloor}\log n~\sim~ \Big(\frac{1}{\mu+\delta}-\frac{b}{\mu}\Big)t\log t, \quad t\to\infty.$$ Further, the assumption concerning the distribution tail of $\xi$ implies that $$\lim_{t\to\infty}\frac{-\log \mmp\{\xi>t\}}{\log t}=\alpha.$$ Hence, for any $\varepsilon\in (0,\alpha)$, all sufficiently large $t$ and all positive integers $n\leq \lfloor (\mu+\delta)^{-1}t\rfloor$, $$(\alpha-\varepsilon)\log (t-\mu n)\leq -\log \mmp\{\xi>t-\mu n\}\leq (\alpha+\varepsilon)\log (t-\mu n).$$ Since
\begin{multline*}
\sum_{n=\lfloor b\mu^{-1}t\rfloor}^{\lfloor (\mu+\delta)^{-1}t \rfloor} \log (t-\mu n)=(\lfloor (\mu+\delta)^{-1}t \rfloor-\lfloor b\mu^{-1}t\rfloor) \log t+\sum_{n=\lfloor b\mu^{-1}t\rfloor}^{\lfloor (\mu+\delta)^{-1}t \rfloor}\log (1-\mu n/t)\\=((\mu+\delta)^{-1}-b\mu^{-1})t\log t+t\int_{b\mu^{-1}}^{(\mu+\delta)^{-1}}\log (1-\mu x){\rm d}x+ o(t),\quad t\to\infty,
\end{multline*}
we conclude that $$\lim_{t\to\infty}\frac{\sum_{n=\lfloor b\mu^{-1}t\rfloor}^{\lfloor (\mu+\delta)^{-1}t \rfloor}(-\log \mmp\{S_n-\mu n>t-\mu n\})}{t\log t}=\Big(\frac{1}{\mu+\delta}-\frac{b}{\mu}\Big)(\alpha-1).$$ It was shown in the proof of Theorem 3.5(b) in~\cite{Alsmeyer+Iksanov+Kabluchko:2025} that
$$\lim_{\delta\to 0+}\limsup_{t\to\infty}\frac{1}{t\log t}\sum_{n=\lfloor (\mu+\delta)^{-1}t\rfloor+1}^{\lfloor \mu^{-1}t \rfloor}(-\log \mmp\{S_{n}>t\})\ =\ 0.$$ The proof of Theorem \ref{thm: regular deviation} is complete.
\end{proof}
\begin{proof}[Proof of Theorem \ref{thm:semi}]
It suffices to prove
$$ \lim_{t\to\infty}\frac{-\log \prod_{n=\lfloor b\mu^{-1}t\rfloor}^{\lfloor \mu^{-1} t\rfloor}\mmp\{S_{n}> t\}}{t^{\varrho+1} \ell(t)}\ =\ \frac{(1-b)^{\varrho+1}}{\mu(\varrho+1)}.$$

As in \cite{Alsmeyer+Iksanov+Kabluchko:2025}, an application of Theorem 2.1 in \cite{Borovkov+Mogulskii:2006} yields
\begin{gather*}
\sum_{n=\lfloor b\mu^{-1}t\rfloor}^{\lfloor \mu^{-1} t\rfloor}\big(-\log \mmp\{S_n-\mu n>t-\mu n\}\big)~\sim~ \sum_{n=\lfloor b\mu^{-1}t\rfloor}^{\lfloor \mu^{-1} t\rfloor} (t-\mu n)^\varrho \ell(t-\mu n)\\
\sim~ \int_{b\mu^{-1}t}^{\mu^{-1}t}(t-\mu x)^\varrho\ell(t-\mu x){\rm d}x\ =\ \frac{1}{\mu}\int_0^{(1-b)t} x^\varrho \ell(x)\ {\rm d}x~\sim~\frac{(1-b)^{\varrho+1}}{\mu(\varrho+1)}t^{\varrho+1}\ell(t),\quad t\to\infty.
\end{gather*}
The proof of Theorem \ref{thm:semi} is complete.
\end{proof}

\begin{proof}[Proof of Theorem \ref{thm: light deviation}]
To prove~\eqref{thm: light deviation_s_claim}, we first write
\begin{multline}\label{eq:thm26_first_order_asymp}
-t^{-2} \log \prod_{n=\lfloor b\mu^{-1}t\rfloor}^{\lfloor \mu^{-1}t\rfloor}\mmp\{S_n>t\}= t^{-2}\int_{\lfloor b\mu^{-1}t\rfloor}^{\lfloor \mu^{-1}t\rfloor}\big(-\log \mmp\{S_{\lfloor y\rfloor}>t\}\big){\rm d}y\\
= \int_{t^{-1}\lfloor b\mu^{-1}t\rfloor}^{t^{-1}\lfloor \mu^{-1}t\rfloor}\frac{-y\log \mmp\{S_{\lfloor ty\rfloor}>t\}}{ty}{\rm d}y.
\end{multline}
The assumptions $\mmp\{\xi\leq x\}<1$ for all $x\geq 0$ and $\me [\eee^{s\xi}]<\infty$ for some $s>0$ ensure that $I(x)<\infty$ for all $x>0$. Indeed, by Markov's inequality, for all $n\in\mn$ and all $x,u>0$, $$-\log \mmp\{\xi>x\}\geq ux-\log \me [\eee^{u\xi}],$$ whence $0\leq I(x)\leq -\log \mmp\{\xi>x\}<\infty$ for all $x>0$. Further, by the Chernoff-Cramer theorem (see, for instance, Theorem 2.1 on p.~116 in \cite{Deheuvels:2007}), for each $y\in [b\mu^{-1}, \mu^{-1}]$,
\begin{equation}\label{eq:inter4}
\lim_{t\to\infty}\frac{-\log \mmp\{S_{\lfloor ty\rfloor}>t\}}{ty}=I(1/y).
\end{equation}
Furthermore, the convergence is uniform in $y\in [b\mu^{-1}, \mu^{-1}]$ by Polya's theorem (see, for instance, p.~113 in \cite{Boas:1996}) because $y\mapsto I(1/y)$ is a continuous function ($I$ is convex) and, for each $t$, the function $-\log \mmp\{S_{\lfloor ty\rfloor}>t\}/(ty)$ is nonincreasing. Hence,
\begin{equation}\label{eq:inter3}
\lim_{t\to\infty} \int_{t^{-1}\lfloor b\mu^{-1}t\rfloor}^{t^{-1}\lfloor \mu^{-1}t\rfloor}\frac{-y\log \mmp\{S_{\lfloor ty\rfloor}>t\}}{ty}{\rm d}y=\int_{b\mu^{-1}}^{\mu^{-1}}yI(1/y){\rm d}y<\infty.
\end{equation}
The proof of~\eqref{thm: light deviation_s_claim} is complete.

To prove~\eqref{thm: light deviation_d_claim}, we apply precise large deviations results both in the linear and moderate settings. First, in view of Theorem 10 on p.~230 in \cite{Petrov:1975} (see also the statement of Theorem 2 on p.~219 in \cite{Petrov:1975} and comments following it), the assumption $\me [\eee^{s\xi}]<\infty$ for some $s>0$ ensures that there is a constant $\beta > 0$ such that
$$\mathbb{P}\{S_n - \mu n > \sigma x\} = \Big(1 - \Phi\Big(\frac{x}{n^{1/2}}\Big)\Big)\exp\Big(\frac{x^3}{n^2} \lambda\Big(\frac{x}{n}\Big)\Big) (1 + \rho\beta)$$ for large $n$ and $x\in [0,\beta n]$, where $\rho=\rho(x,n,\beta)$ is bounded by an absolute constant, $\lambda$ is a power series with a positive radius of convergence, called Cram\'{e}r's series, $\sigma^2={\rm Var}\,[\xi]\in (0,\infty)$, and $$y\mapsto \Phi(y):=\frac{1}{\sqrt{2\pi}}\int_{-\infty}^{y}\eee^{-s^2/2}{\rm d}s,\quad y\in\mr$$ is the standard normal distribution function.

As a consequence, for a sufficiently small $\varepsilon>0$ satisfying
$b<\mu/(\mu+\varepsilon)$,
\begin{align*}
&\hspace{-0.3cm}- \log \prod_{n=\lfloor t/(\mu+\varepsilon) \rfloor +1}^{\lfloor t/\mu \rfloor} \mathbb{P}\{S_n > t\}
= - \log \prod_{n=\lfloor t/(\mu+\varepsilon) \rfloor +1}^{\lfloor t/\mu \rfloor}  \mathbb{P}\{S_n - \mu n > t - \mu n\}\\
&= - \sum_{n=\lfloor t/(\mu+\varepsilon) \rfloor +1}^{\lfloor t/\mu \rfloor} \log \Big(1 - \Phi\Big(\frac{t - \mu n}{\sigma \sqrt{n}}\Big)\Big)
- \sum_{n=\lfloor t/(\mu+\varepsilon) \rfloor +1}^{\lfloor t/\mu \rfloor} \frac{(t - \mu n)^3}{\sigma^3 n^2} \lambda\Big(\frac{t - \mu n}{\sigma n}\Big) + O(t)\\
& =I + II + O(t).
\end{align*}

We shall need an asymptotic expansion of $\Phi$:
$$1 - \Phi(x) = \frac{\eee^{-x^2/2}}{(2\pi)^{1/2}x} \big(1 + O(x^{-2})\big),\quad x \to \infty.$$ Fix some $\delta \in (0, \frac{1}{6})$ and put $t_1:=\lfloor {t}/(\mu+\varepsilon) \rfloor + 1$ and $t_2:= \lfloor {t}/{\mu} - t^{1/2+\delta} \rfloor$. For large $t$, we further split $I$ depending on whether $t-\mu n\leq \mu t^{1/2+\delta}$ or $t-\mu n>\mu t^{1/2+\delta}$: $$I=\sum_{n=t_2+1}^{\lfloor t/\mu\rfloor}\ldots+\sum_{n=t_1}^{t_2}\ldots=:I_1+I_2.$$ We proceed by estimating $I_1$:
\begin{multline*}
I_{1} = - \sum_{n=t_2+1}^{\lfloor t/\mu \rfloor} \log \Big(1 - \Phi\Big(\frac{t - \mu n}{\sigma n^{1/2}}\Big)\Big)
\le -(\lfloor t/\mu\rfloor-t_2) 
\log \Big(1 - \Phi\Big (\frac{\mu t^{1/2+\delta} }{\sigma t_2^{1/2}}\Big)\Big)~\sim~\frac{\mu^3}{2\sigma^2}t^{1/2+3\delta}\\= o(t),\quad t\to\infty
\end{multline*}
having utilized the fact that the function $x\mapsto -\log (1-\Phi(x))$ is increasing on $(0,\infty)$. As for the second fragment we infer
\begin{align*}
I_{2}&= - \sum_{n=t_1}^{t_2} \log \Big(1 - \Phi\Big(\frac{t - \mu n}{\sigma n^{1/2}}\Big)\Big)\\
&= \sum_{n=t_1}^{t_2} \frac{1}{2} \Big(\frac{t - \mu n}{\sigma n^{1/2}}\Big)^2 + \sum_{n=t_1}^{t_2} \log \Big(\frac{t - \mu n}{n^{1/2}}\Big)+ \sum_{n=t_1}^{t_2}  \log (2\pi\sigma^{-2})^{1/2}+\sum_{n=t_1}^{t_2} \log \Big(1 + O\Big(\frac{n}{(t - \mu n)^2}\Big)\Big)\\
&= t^2 \cdot \frac{1}{2\sigma^2t } \sum_{n=t_1}^{t_2} \frac{t}{n} \Big(1 - \frac{\mu n}{t}\Big)^2 + \sum_{n=t_1}^{t_2} \log \Big(\frac{t - \mu n}{n^{1/2}}\Big)+\sum_{n=t_1}^{t_2}  \log (2\pi\sigma^{-2})^{1/2}+
\sum_{n=t_1}^{t_2}  O(t^{-2\delta})\\
&= t^2 \cdot \frac{1}{2\sigma^2t } \sum_{n=t_1}^{\lfloor t/\mu\rfloor} \frac{t}{n} \Big(1 - \frac{\mu n}{t}\Big)^2 +O(t^{\frac{1}{2}+3\delta})+ \sum_{n=t_1}^{t_2} \log \Big(\frac{t - \mu n}{n^{1/2}}\Big)+\sum_{n=t_1}^{t_2}  \log \left(\frac{2\pi}{\sigma^{2}}\right)^{\frac{1}{2}}+\sum_{n=t_1}^{t_2}  O(t^{-2\delta})\\
&=: t^2 \cdot I_{2,1} + \sum_{n=t_1}^{t_2} \log \Big(\frac{t - \mu n}{n^{1/2}}\Big)+ O(t),
\end{align*}
where the penultimate equality follows from
$$
\sum_{n=t_2+1}^{\lfloor t/\mu\rfloor} \frac{t}{n} \Big(1-\frac{\mu n}{t}\Big)^2\leq \frac{1}{t_2 t} \sum_{n=t_2+1}^{\lfloor t/\mu\rfloor} (t - \mu n)^2\leq \frac{\lfloor t/\mu\rfloor-t_2}{t_2 t} (t - \mu t_2)^2=O(t^{-1/2+3\delta}).
$$
For each $t>0$, put $g_t(x):=(t/x)(1-\mu x/t)^2$, $x>0$. The function $g_t$ is decreasing on $[t/(\mu+\varepsilon), t/\mu]$. Hence,
\begin{multline*}
2\sigma^2 tI_{2,1}=\sum_{n=t_1}^{\lfloor t/\mu\rfloor}g_t(n)=\int_{t/(\mu+\varepsilon)}^{t/\mu}g_t(x){\rm d}x+O(g_t(t/(\mu+\varepsilon)))+O(g_t(t/\mu))\\=t\int_{1/(\mu+\varepsilon)}^{1/\mu} y^{-1}(1-\mu y)^2{\rm d}y+O(1),\quad t\to\infty.
\end{multline*}
Further, we claim that
\begin{equation}\label{eq:stirling_proof1}
\sum_{n=t_1}^{t_2} \log \Big(\frac{t - \mu n}{n^{1/2}}\Big)=\sum_{n=t_1}^{t_2} \log(t - \mu n)-\frac{1}{2}\sum_{n=t_1}^{t_2} \log n=\frac{\varepsilon}{2 \mu (\mu + \varepsilon)}t \log t+O(t),\quad t\to\infty.
\end{equation}
Indeed, by Stirling's formula, $$\sum_{n=t_1}^{t_2} \log n=\frac{\varepsilon}{\mu (\mu + \varepsilon)}t \log t+O(t),\quad t\to\infty.$$ The first summand in the penultimate formula admits the same expansion. By monotonicity, it suffices to prove this fact with $t=\mu k$ for $k\in\mn$. Another application of Stirling's formula yields
\begin{multline*}
\sum_{n=\lfloor (\mu k)/(\mu+\varepsilon)\rfloor+1}^{\lfloor k-(\mu k)^{1/2+\delta}\rfloor }\log (\mu k - \mu n)=\sum_{n=-\lfloor -(\mu k)^{1/2+\delta}\rfloor}^{k-\lfloor (\mu k)/(\mu+\varepsilon)\rfloor-1}\log n +O(k)=\frac{\varepsilon}{\mu + \varepsilon}k\log k+O(k)\\
=\frac{\varepsilon}{\mu(\mu + \varepsilon)}t\log t+O(t),\quad t\to\infty.
\end{multline*}
Thus,~\eqref{eq:stirling_proof1} has been justified. Summarizing
\begin{equation}
\label{eq:w1}
I = t^2 \cdot \frac{1}{2\sigma^2} \int_{1/(\mu+\varepsilon) 
}^{1/\mu} \frac{1}{y} (1 - \mu y)^2 {\rm d}y + \frac{\varepsilon}{2\mu (\mu + \varepsilon)} t \log t + O(t),\quad t\to\infty.
\end{equation}
To estimate $II$, observe that decreasing $\varepsilon>0$ if necessary, we can assume that the function  $y\mapsto \frac{(1-\mu y)^3}{y^2}\lambda(1/(\sigma y)-\mu/\sigma)$ has a bounded derivative on $[1/(\mu + \varepsilon), 1/\mu]$. Thus, approximating the Riemann sum by the corresponding integral yields
\begin{multline}\label{eq:w2}
II = -\frac{1}{\sigma^3 } \sum_{n=t_1}^{\lfloor t/\mu\rfloor} \frac{(t-\mu n)^3}{n^2} \lambda\bigg(\frac{t}{\sigma n} - \frac{\mu}{\sigma}\bigg)=-\frac{t^2}{\sigma^3}\frac{1}{t}\sum_{n=t_1}^{\lfloor t/\mu\rfloor} \frac{(1-\mu n/t)^3}{(n/t)^2} \lambda\bigg(\frac{t}{\sigma n} - \frac{\mu}{\sigma}\bigg)\\
= -\frac{t^2}{\sigma^3} \int_{1/(\mu + \varepsilon)}^{1/\mu} \frac{(1-\mu y)^3}{y^2} \lambda\bigg(\frac{1}{\sigma y} - \frac{\mu}{\sigma}\bigg){\rm d}y + O(t),\quad t\to\infty.
\end{multline}

Recall that $\varepsilon>0$ is chosen such that $b<\mu/(\mu+\varepsilon)$. We now estimate
$$
- \log  \prod_{n = \lfloor b\mu^{-1}t  \rfloor}^{\lfloor  t/(\mu+\varepsilon) \rfloor}
 \mathbb{P}\{S_n > t\}.
$$
To this end, we invoke a precise large deviation result proved in Theorem 1 of~\cite{Bahadur:Rao:1960} and Theorem 1 of~\cite{Petrov:1965}. If $\theta \in (\mu+\varepsilon, A_0-\varepsilon)$, then
$$\mathbb{P}\{S_n > \theta n\} = \frac{\eee^{-nI(\theta)}}{a
\nu(2\pi n)^{1/2}} (1 + o(1)),\quad n\to\infty,$$
where $a$ is a parameter such that $(\log \Lambda)^\prime(a)
= \theta$ and $\nu^2 = (\log \Lambda)^{\prime\prime}(a)$. Given $n$ and $t$, let $a_n$ denote a parameter satisfying $(\log \Lambda)^\prime(a_n) = t/n=:\theta_n$ and put $\nu_n^2:= (\log \Lambda)^{\prime\prime}(a_n)$. Then the term $o(1)$ is uniformly small (see Theorem 1 in \cite{Petrov:1965}) and all the parameters $a_n$ and $\nu_n$ are bounded away from $0$ and $\infty$, uniformly over $n$ such that $n/t$ ranges over a compact subset of $(1/A_0,1/\mu)$. In view of \eqref{eq:inter4}, for each $t>0$, the function $x\mapsto (x/t)I(t/x)$ is nonincreasing on $[b\mu^{-1}t, \mu^{-1}t]$ as the pointwise limit of nonincreasing functions. This ensures that $$\sum_{n = \lfloor b\mu^{-1}t  \rfloor}^{\lfloor  t/(\mu+\varepsilon) \rfloor }(n/t)I(t/n)=t\int_{b/\mu}^{1/(\mu + \varepsilon)}yI(1/y){\rm d}y+O(1),\quad t\to\infty.$$ Using this together with Stirling's formula we infer
\begin{align*}
- \log  \prod_{n = \lfloor b\mu^{-1}t  \rfloor}^{\lfloor  t/(\mu+\varepsilon) \rfloor }
 \mathbb{P}\{S_n > t\}&= \sum_{n = \lfloor b\mu^{-1}t  \rfloor}^{\lfloor  t/(\mu+\varepsilon) \rfloor } n I(\theta_n) + \frac{1}{2} \sum_{n = \lfloor b\mu^{-1}t  \rfloor}^{\lfloor  t/(\mu+\varepsilon) \rfloor 
 } \log n
 -\sum_{n = \lfloor b\mu^{-1}t  \rfloor}^{\lfloor  t/(\mu+\varepsilon) \rfloor }\log \bigg( \frac{1+o(1)}{a_n \nu_n \sqrt{2\pi}}\bigg)\\
&= t^2 \cdot \frac{1}{t} \sum_{n = \lfloor b\mu^{-1}t  \rfloor}^{\lfloor  t/(\mu+\varepsilon) \rfloor }  \frac{n}{t} I(t/n) + \frac{1}{2} \sum_{n = \lfloor b\mu^{-1}t  \rfloor}^{\lfloor  t/(\mu+\varepsilon) \rfloor 
} \log n + O(t)\\
&= t^2\cdot \int_{b/\mu}^{1/(\mu + \varepsilon)} y I(1/y) {\rm d}y + \frac{1}{2} \bigg(\frac{1}{\mu + \varepsilon} - \frac{b}{\mu}\bigg) t \log t + O(t),\quad t\to\infty.
\end{align*}
Combining pieces together we obtain
\begin{multline*}
-\log \prod_{n=\lfloor b\mu^{-1}t\rfloor}^{\lfloor\mu^{-1}t \rfloor}\mmp\{S_n>t\}=t^2\cdot\bigg(\frac{1}{2\sigma^2} \int_{1/(\mu+\varepsilon)}^{1/\mu} \frac{1}{y} (1 - \mu y)^2 {\rm d}y+\int_{b/\mu}^{1/(\mu + \varepsilon)} y I(1/y) {\rm d}y\\ -\frac{1}{\sigma^3}\int_{1/(\mu + \varepsilon)}^{1/\mu} \frac{(1-\mu y)^3}{y^2} \lambda\Big(\frac{1}{\sigma y} - \frac{\mu}{\sigma}\Big){\rm d}y\bigg) + \frac{1-b}{2\mu}  t \log t + O(t).
\end{multline*}
On the other hand, according to the first order asymptotic secured by~\eqref{eq:thm26_first_order_asymp} and~\eqref{eq:inter3}, we must have
\begin{multline*}
\frac{1}{2\sigma^2} \int_{1/(\mu+\varepsilon)}^{1/\mu} \frac{1}{y} (1 - \mu y)^2 {\rm d}y+\int_{b/\mu}^{1/(\mu + \varepsilon)} y I(1/y) {\rm d}y\\ -\frac{1}{\sigma^3}\int_{1/(\mu + \varepsilon)}^{1/\mu} \frac{(1-\mu y)^3}{y^2} \lambda\Big(\frac{1}{\sigma y} - \frac{\mu}{\sigma}\Big){\rm d}y=\int_{b/\mu}^{1/\mu}yI(1/y){\rm d}y,
\end{multline*}
for any sufficiently small $\varepsilon>0$. Thus, relation \eqref{thm: light deviation_d_claim} does indeed hold true. The proof of Theorem \ref{thm: light deviation} is complete.
\end{proof}

In view of Lemma \ref{lem:shirai} and formula \eqref{eq:appr2}, it is enough to check that relations \eqref{thm: light deviation2_s_claim} and \eqref{thm: light deviation2_d_claim} stated in Theorem \ref{thm: light deviation2} hold with $-\log \prod_{n=\lfloor \mu^{-1} t\rfloor+1}^{\lfloor b\mu^{-1}t\rfloor}\mmp\{S_n \leq t\}$ replacing $-\log \mmp\{\hat N(t)=\lfloor b\mu^{-1}t\rfloor \}$.

\begin{proof}[Proof of Theorem~\ref{thm: light deviation2}]
As far as ~\eqref{thm: light deviation2_s_claim} is concerned, it suffices to show that
$$
\lim_{t\to\infty}\frac{-\log \prod_{n=\lfloor \mu^{-1}t\rfloor}^{\lfloor b \mu^{-1} t\rfloor}\mmp\{S_n\leq t\}}{t^2}=\int_{\mu^{-1}}^{b\mu^{-1}}yI^\ast(1/y){\rm d}y<\infty.
$$
The argument behind proving this limit relation mimics that exploited in the proof of~\eqref{thm: light deviation_s_claim} in Theorem \ref{thm: light deviation}. The only difference is that now we invoke a version of the Chernoff-Cramer theorem for the left distribution tail (see again Theorem 2.1 on p.~116 in \cite{Deheuvels:2007}) saying that, for each $y\in [\mu^{-1}, b\mu^{-1}]$,
\begin{equation}\label{eq:inter2}
\lim_{t\to\infty}\frac{-\log \mmp\{S_{\lfloor ty\rfloor}\leq t\}}{ty}=I^\ast(1/y).
\end{equation}
Finiteness of the right-hand side is secured by $0\leq I^\ast(1/y)\leq -\log \mmp\{\xi\leq 1/y\}<\infty$ for $y\in [\mu^{-1}, b\mu^{-1}]$, where the last inequality is justified by the assumption $\mmp\{\xi\leq b^{-1}\mu\}>0$. We omit further details.

The second claim~\eqref{thm: light deviation2_d_claim} can be proved by 
the same reasoning as in the proof of~\eqref{thm: light deviation_d_claim} but applied to the random walk $(-S_n)_{n\geq 1}$ with the generic step $-\xi$. This also explains the appearance of the function $I^{\ast}$ derived from the moment generating function of the distribution of $-\xi$ in~\eqref{thm: light deviation2_s_claim} and~\eqref{thm: light deviation2_d_claim}.
\end{proof}

\subsection{Proof of Proposition \ref{thm:local limit}}

We only need to prove relation \eqref{eq:diver}. Once this is done, formula \eqref{eq:local} is secured by Proposition \ref{prop:Dolg}. Recall the notation $\nu(t)=\inf\{k\geq 1: S_k>t\}$ for $t\geq 0$ and observe that
$${\rm Var}\,[\hat N(t)]=\sum_{k\geq 1}\mmp\{\nu(t) > k\}\mmp\{\nu(t)\leq k\},\quad t\geq 0.$$ For $t>0$, let $m(t)$ denote the median of the distribution of $\nu(t)$, that is, any positive integer satisfying
$$\mmp\{\nu(t) \geq m(t)\}\geq 1/2\quad\text{and}\quad\mmp\{\nu(t) \leq m(t)\}\geq 1/2.$$ Plainly, $\lim_{t\to\infty}\nu(t)=\infty$ a.s.\ entails $\lim_{t\to\infty}m(t)=\infty$. Let $a:(0,\infty)\to\mn$ be a function to be specified later. We estimate ${\rm Var}\,[\hat N(t)]$ from below:
\begin{align*}
{\rm Var}\,[\hat N(t)]&\geq \sum_{k=m(t)}^{m(t)+a(t)-1}\mmp\{\nu(t) > k\}\mmp\{\nu(t)\leq k\}\geq \mmp\{\nu(t)\leq m(t)\}\sum_{k=m(t)}^{m(t)+a(t)-1}\mmp\{\nu(t) > k\}\\
&\geq\frac{1}{2}a(t)\mmp\{\nu(t)> m(t)+a(t)-1\}=\frac{1}{2}a(t)\mmp\{S_{m(t)+a(t)-1}\leq t\}.
\end{align*}
Thus,~\eqref{eq:diver} follows once we can find $a$ such that
\begin{equation}\label{eq:proof_of_var_divergencea1}
\lim_{t\to\infty}a(t)=\infty\quad\text{and}\quad \limsup_{t\to\infty}\mmp\{S_{m(t)+a(t)-1}> t\}<1.
\end{equation}
In what follows, let $t>0$ be large enough so that $m(t)\geq 2$. Note that
\begin{align*}
&\hspace{-1cm}\mmp\{S_{m(t)+a(t)-1}> t\}=\int_{[0,\,\infty)}\mmp\{t-y<S_{m(t)-1}\}\mmp\{S_{a(t)}\in {\rm d}y\}\\
&=\mmp\{S_{m(t)-1}> t\}+\int_{[0,\,\infty)}\mmp\{t-y<S_{m(t)-1}\leq t\}\mmp\{S_{a(t)}\in {\rm d}y\}\\
&= 1-\mmp\{\nu(t)\geq m(t)\}+\int_{[0,\,\infty)}\mmp\{t-y<S_{m(t)-1}\leq t\}\mmp\{S_{a(t)}\in {\rm d}y\}\\
&\leq 1/2+\int_{[0,\,\infty)}\mmp\{t-y\leq S_{m(t)-1}\leq t\}\mmp\{S_{a(t)}\in {\rm d}y\}.
\end{align*}
The integrand on the right-hand side can be estimated with the help of Theorem 9 on p.~49 in~\cite{Petrov:1975}. The cited result implies that, for some absolute constant $C>0$ (which depends on the distribution of $\xi$),
\begin{align*}
\mmp\{S_{m(t)+a(t)-1}> t\}&\leq 1/2+C\int_{[0,\,\infty)}\left(\frac{y+1}{(m(t)-1)^{1/2}}\wedge 1\right)\mmp\{S_{a(t)}\in {\rm d}y\}\\
&=1/2+C\,\mathbb{E}\left[\frac{S_{a(t)}+1}{(m(t)-1)^{1/2}}\wedge 1\right]\\
&\leq 1/2+C\sum_{k=1}^{a(t)}\mathbb{E}\left[\frac{\xi_k}{(m(t)-1)^{1/2}}\wedge 1\right]+\frac{C}{(m(t)-1)^{1/2}}\\
&= 1/2+C\,a(t)\mathbb{E}\left[\frac{\xi}{(m(t)-1)^{1/2}}\wedge 1\right]+\frac{C}{(m(t)-1)^{1/2}}.
\end{align*}
Here, $x\wedge 1=\min (x,1)$ for $x\geq 0$, and we have used subadditivity of $x\mapsto x\wedge 1$ on $[0,\infty)$ for the last inequality. Using $\lim_{t\to\infty}m(t)=\infty$ and Lebesgue's dominated convergence theorem, we conclude that
$$
\lim_{t\to\infty}\mathbb{E}\left[\frac{\xi}{(m(t)-1)^{1/2}}\wedge 1\right]=0.
$$
Hence, we can choose $a$ satisfying $\lim_{t\to\infty}a(t)=\infty$ and $$\lim_{t\to\infty}a(t)\mathbb{E}\left[\frac{\xi}{(m(t)-1)^{1/2}}\wedge 1\right]=0.$$ With this choice, the $\limsup$ in~\eqref{eq:proof_of_var_divergencea1} does not exceed $1/2$. This finishes the proof of both~\eqref{eq:diver} and Proposition~\ref{thm:local limit}.

\section{Appendix}

Here, we formulate a local central limit theorem for triangular arrays of independent Bernoulli random variables.

Let $(\eta_{k,t})_{k\geq 1, t\geq 0}$ be a family of Bernoulli random variables defined on a common probability space. We assume that, for each $t\geq 0$, the random variables $\eta_{1,t}$, $\eta_{2,t},\ldots$ are independent. Put $$Y(t):=\sum_{k\geq 1}\eta_{k,t},\quad t\geq 0$$ and, for each $k\geq 1$ and each $t\geq 0$, put $p_{k,t}:=\mmp\{\eta_{k,t}=1\}$. Fix any $t_0>0$. By the Borel-Cantelli lemma, the series defining $Y(t_0)$ converges a.s.\ if, and only if, $\sum_{k\ge 1}p_{k,t_0}<\infty$. On the other hand, if $\sum_{k\ge 1}p_{k,t_0}=\infty$, then the series diverges a.s. We assume in what follows that $\sum_{k\ge 1}p_{k,t}<\infty$ for all $t\geq 0$. In particular, $\me [Y(t)]=\sum_{k\geq 1}p_{k,t}<\infty$ and ${\rm Var}\,[Y(t)]=\sum_{k\geq 1}p_{k,t}(1-p_{k,t})\leq \me [Y(t)]<\infty$ for all $t\geq 0$.

The following proposition is an extension of Theorem 2 in~\cite{Bender:1973} to infinite sums of independent Bernoulli random variables. Although the result in~\cite{Bender:1973} is stated for finite sums, its proof carries over to the present setting. Indeed, the log-concavity of the probability mass function, required in Lemma 2 of that paper, is preserved upon passing to the limit from finite sums.

\begin{proposition}\label{prop:Dolg}
Assume that $\lim_{t\to\infty}{\rm Var}\,[Y(t)]=\infty$. Then $$\lim_{t\to\infty}\sup_{k\geq 0}\Big|({\rm Var}\,[Y(t)])^{1/2}\mmp\{Y(t)=k\}-\frac{1}{(2\pi)^{1/2}}\exp\Big(-\Big(\frac{k-\me [Y(t)]}{({\rm Var}\,[Y(t)])^{1/2}}\Big)^2\Big/2\Big)\Big|=0.$$ In particular, $$\mmp\{Y(t)=\lfloor \me [Y(t)]\rfloor\}~\sim~\frac{1}{(2\pi {\rm Var}\,[Y(t)])^{1/2}},\quad t\to\infty.$$
\end{proposition}

\noindent \textbf{Acknowledgment.} We thank the Associate Editor and the referee for their helpful comments. D.~Buraczewski was supported by the National Science Center, Poland (Opus, grant number 2020/39/B/ST1/00209). A.~Iksanov was supported by the University of Wroc{\l}aw in the framework of the programme `Initiative of Excellence~--~Research University' (IDUB).

\end{document}